\documentclass[10pt,a4paper,final]{article}
\usepackage[english]{babel}%
\usepackage{amssymb,amsmath,ae,amsthm}%
\usepackage[dvips]{graphicx}%
\theoremstyle{plain}
\newtheorem{theorem}{Theorem}[section]
\newtheorem{corollary}[theorem]{Corollary}
\newtheorem{lemma}[theorem]{Lemma}

\newtheorem{problem}[theorem]{Problem}

\theoremstyle{definition}
\newtheorem{definition}[theorem]{Definition}
\newtheorem{assumption}[theorem]{Assumption}
\newtheorem{remark}[theorem]{Remark}

\newcommand{\upidx}[1]{{\at{#1}}}
\newcommand{\xspAux}{{X_{\aux}}}
\newcommand{\sspAux}{{Y_{\aux}}}
\newcommand{\fspAux}{{\mathcal{Y}_{\aux}}}

\newcommand{\sspApp}{{Y_{\app}}}
\newcommand{\fspApp}{{\mathcal{Y}_{\app}}}
\newcommand{\sspttApp}{{\widetilde{Y}_{\app}}}
\newcommand{\fspttApp}{{\widetilde{\mathcal{Y}}_{\app}}}
\newcommand{\itopApp}{{\mathcal{I}_{\app}}}

\newcommand{\xspInf}{{X}}
\newcommand{\sspInf}{{Y}}
\newcommand{\fspInf}{{\mathcal{Y}}}

\newcommand{\abs}[1]{{\left|#1\right|}}
\newcommand{\bigabs}[1]{{\big|}#1{\big|}}

\newcommand{\babs}[1]{\abs{#1}_{\beta}}
\newcommand{\bnorm}[1]{\norm{#1}_{\beta}}
\newcommand{\lnorm}[1]{\norm{#1}_{\mathrm{mom}}}
\newcommand{\bpnorm}[2]{{\left\lfloor#1\right\rfloor}_{#2,\,\beta}}

\newcommand{\bigpar}{$\;$\par\quad\par\noindent}
\newcommand{\dint}{\mathrm{d}}

\newcommand{\deq}{:=}

\newcommand{\triple}[3]{{\left({#1},{#2},{#3}\right)}}
\newcommand{\quadruple}[4]{{\left({#1},{#2},{#3},{#4}\right)}}
\newcommand{\pair}[2]{{\left({#1},{#2}\right)}}
\newcommand{\npair}[2]{{({#1},{#2})}}
\newcommand{\at}[1]{{\left({#1}\right)}}
\newcommand{\bat}[1]{{\big({#1}\big)}}
\newcommand{\Bat}[1]{{\Big({#1}\Big)}}
\newcommand{\nat}[1]{{({#1})}}
\newcommand{\eps}{\varepsilon}
\newcommand{\norm}[1]{{||}#1{||}}

\newcommand{\ol}[1]{\overline{#1}}

\newcommand{\aux}{\mathrm{aux}}

\newcommand{\app}{\mathrm{app}}

\newcommand{\Rset}{{\mathbb{R}}}
\newcommand{\Nset}{{\mathbb{N}}}

\newcommand{\cinterval}[2]{[#1,\,#2]}%
\newcommand{\cointerval}[2]{[#1,\,#2)}%
\newcommand{\ccinterval}[2]{[#1,\,#2]}%

\newcommand{\calT}{{\mathcal{T}}}
\newcommand{\calN}{{\mathcal{N}}}
\newcommand{\calS}{{\mathcal{S}}}

\newcommand{\calY}{{\mathcal{Y}}}
\newcommand{\calE}{{\mathcal{E}}}

\begin{document}
\title{A Kinetic Model for Grain Growth}

\author{R.~Henseler \and M.~Herrmann \and  \and J.~J.~L.~Vel{\'a}zquez}

\author{ %
    R.~Henseler\thanks{%
        Universit{\"a}t Bonn,\,Germany,\,%
        {\tt henseler@uni-bonn.de }}
    \and
    M.~Herrmann\thanks{%
        University of Oxford,\,England,\,%
        {\tt michael.herrmann@maths.ox.ac.uk}}
    \and
    B.~Niethammer\thanks{%
        University of Oxford,\,England,\,%
        {\tt niethammer@maths.ox.ac.uk}}
    \and
    J.~J.~L.~Vel{\'a}zquez\thanks{%
        Instituto de Ciencias Mathem\'{a}ticas (CSIC-UAM-UC3M-UCM),\,Spain,\,%
        {\tt jj{\_}velazquez@mat.ucm.es}}
   }%
\date{\today}%
\maketitle%
%
%
%
%
\begin{abstract}
We provide a well--posedness analysis of a kinetic model for grain
growth introduced by Fradkov which is based on the {\it von
Neumann--Mullins law}. The model consists of an infinite number of
transport equations with a tri-diagonal coupling modelling
topological changes in the grain configuration. Self--consistency of
this kinetic model is achieved by introducing a coupling weight
which leads to a nonlinear and nonlocal system of equations.
\par
We prove existence of solutions by approximation with 
finite dimensional systems. Key ingredients in passing to the limit
are suitable super--solutions, a bound from below on the total mass,
and a tightness estimate which ensures that no mass is transported
to infinity in finite time.
\end{abstract}
\quad\newline%
\begin{minipage}{\textwidth}
\end{minipage}
\quad\newline%
\begin{minipage}{\textwidth}
\end{minipage}
%

%
%
%
%
%
%
%
%
%
%
\section{Introduction}
%
%
Many technologically useful materials are poly-crystalline
aggregates, composed of a huge number of crystallites, called
grains, separated by so--called grain boundaries. Typically such
materials undergo an aging process leading to
coarsening of the grain structure and therefore inducing changes
in mechanical, electrical, optical, and magnetic properties of the
material. For details we refer to the review articles by Fradkov and
Udler \cite{FU94} and by Thompson \cite{T01}.
\par
Different approaches for modelling grain growth in two space
dimensions are established in the literature. In Monte--Carlo
models, compare \cite{ASGS84a,ASGS84b}, the kinetics of the boundary
motion are simulated by employing a Monte--Carlo technique for
moving discrete lattice points describing the microstructure. An
attractive feature of this model is the simple handling of
topological events like grain boundary flipping and grain
disappearance.
\par
Boundary tracking models based on partial differential equations as
discussed in \cite{KL01,MNT04} offer an alternative to Monte--Carlo
models as they deal with quantities of lower dimension. They can be
further reduced to so--called vertex models where movement of grain
boundaries is projected onto the triple--junctions, see
\cite{KNN89,HNO03}. In both cases, however, topological changes
require extra treatment.
\par %
In the sequel we
focus on a kinetic mean-field type models \cite{Fr88a,M87,F93} based on
the {\it von Neumann--Mullins law}. Such models consider
time--dependent distribution functions for the grain areas and the
number of sides per grain. Grain areas change according to the {\it
von Neumann--Mullins law}, topological changes are modelled by
collision-type operators. Fradkov\cite{Fr88a} was the first to
develop a model of this type which consists of an infinite--dimensional
system of transport equations with a nonlocal weight, making the
equations nonlinear. In this article we establish a rigorous
well--posedness theory for this model.

%
%
\section{The model}\label{sec:model}
%
%
In this section we present a derivation of Fradkov's kinetic model
for grain growth which is based on the {\it von Neumann--Mullins law} for the
change of grain areas and the so--called {\it `gas'
approximation} to describe topological changes in a 2D network of grains.
Starting point is the isotropic mean--curvature flow for the grain boundaries
with equilibrium of forces at triple junctions.
%
%
\subsection{Networks of grains with triple junctions}
\paragraph*{Motion by mean curvature and equilibrium of forces at triple
junctions}
%
Mean curvature flow coupled with equilibrium of forces at triple
junctions is a widely accepted model for two--dimensional grain
growth \cite{BR93,KL01,MNT04}. For simplicity our objects are
1--periodic spatial networks whose curves meet in triple junctions
(Fig. \ref{fig:isocrystals}). We restrict ourselves in the following to the case of isotropic
surface energies, such that the curves move according to the
isotropic mean curvature flow. Moreover, we assume that the mobility
of the triple junctions is infinite compared to the mobility of the
grain boundaries, and this leads to equilibrium of forces at triple
junctions. In the isotropic case this condition, also known as {\it
Herring condition}, just means that the curves meet in an angle of
$2\pi/3$. The Herring condition also arises as the natural boundary
condition in the interpretation of the mean curvature flow of
networks as $L^2$--gradient flow of the surface energy, see
\cite{TC94,HNO03}.
\begin{figure}[h!]
\centering%
{\includegraphics[width=0.25\textwidth]{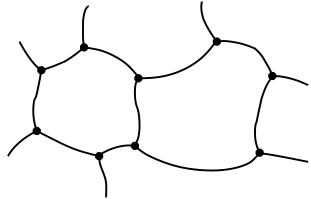}}%
\caption{Cartoon of a 2D network of grains with triple junctions
indicated by bullets. The Herring condition implies all angles
to equal $2\pi/3$.} \label{fig:isocrystals}
\end{figure}
%
%
\paragraph*{Von Neumann--Mullins law}
%
Under the assumptions stated above (isotropic surface energy, equal
mobility of grain boundaries, and infinite mobility of triple
junctions), one can  derive a law of motion for the area
$a\left(t\right)$ of a single grain with $n$ edges \cite{M56}, known
as the {\it von Neumann--Mullins law}:
\begin{align}
\label{vNMl} \frac{\dint{}}{\dint{t}} a\left(t\right) = M \sigma
\frac{\pi}{3} \left(n-6\right)
\end{align}
Here $M$ denotes the mobility of the grain boundaries and $\sigma$
the surface tension. The proof uses a direct geometric computation
involving motion by curvature of the grain boundaries and the
prescribed jumps of the outer normal by $2\pi/3$ at triple
junctions.\par The {\it von Neumann--Mullins law} implies that
grains with less than six edges shrink, those with more than six
grow, and such with exactly six edges retain their area (possibly
not their shape).
%
%
\paragraph*{Topological changes}
%
The evolution by mean curvature is well--defined until two vertices
on a grain boundary collide, after which topological rearrangements
may take place. This happens when either an edge or a whole grain
vanishes. In the first case an unstable fourfold vertex is produced,
which immediately splits up again, usually in such a way that two
new vertices are connected by a new edge.
\begin{figure}[ht]
\centering%
{\includegraphics[width=0.4\textwidth]{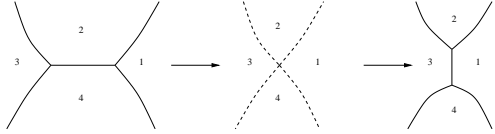}}%
\caption{Neighbour switching} \label{fig:switching}
\end{figure}
In this case, two neighbouring grains decrease their topological
class (i.e. the number of edges), whereas the two other grains
increase it (Fig. \ref{fig:switching}). The second case causing
topological rearrangements is grain vanishing. Each grain vanishing
is accompanied by disappearance of two vertices and three edges. Due
to the {\it von Neumann--Mullins law} we only take grains with
topological class $2 \le n \le 5$ into account.
\begin{figure}[ht]
\centering%
{%
\begin{minipage}[c]{0.4\textwidth}%
\includegraphics[width=\textwidth]{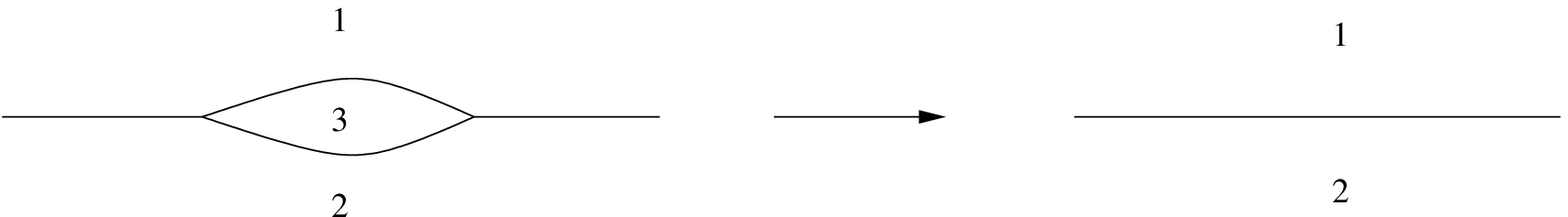}%
\end{minipage}%
\hspace{0.1\textwidth}%
\begin{minipage}[c]{0.4\textwidth}%
\includegraphics[width=\textwidth]{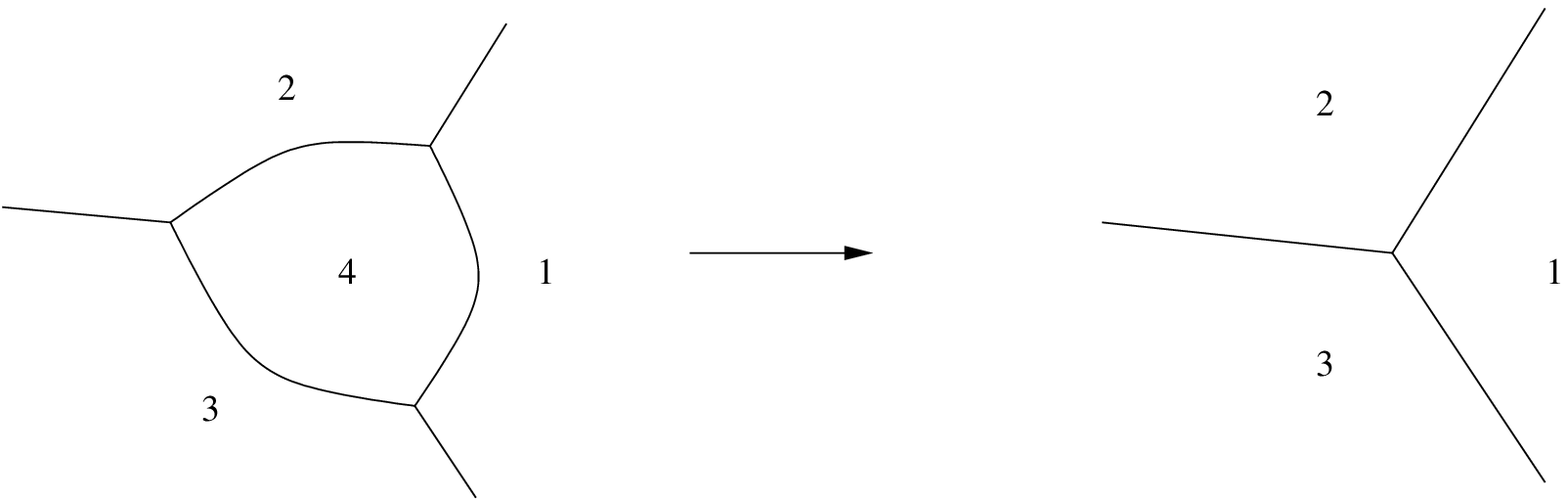}%
\end{minipage}%
\\%
\begin{minipage}[c]{0.4\textwidth}%
\includegraphics[width=\textwidth]{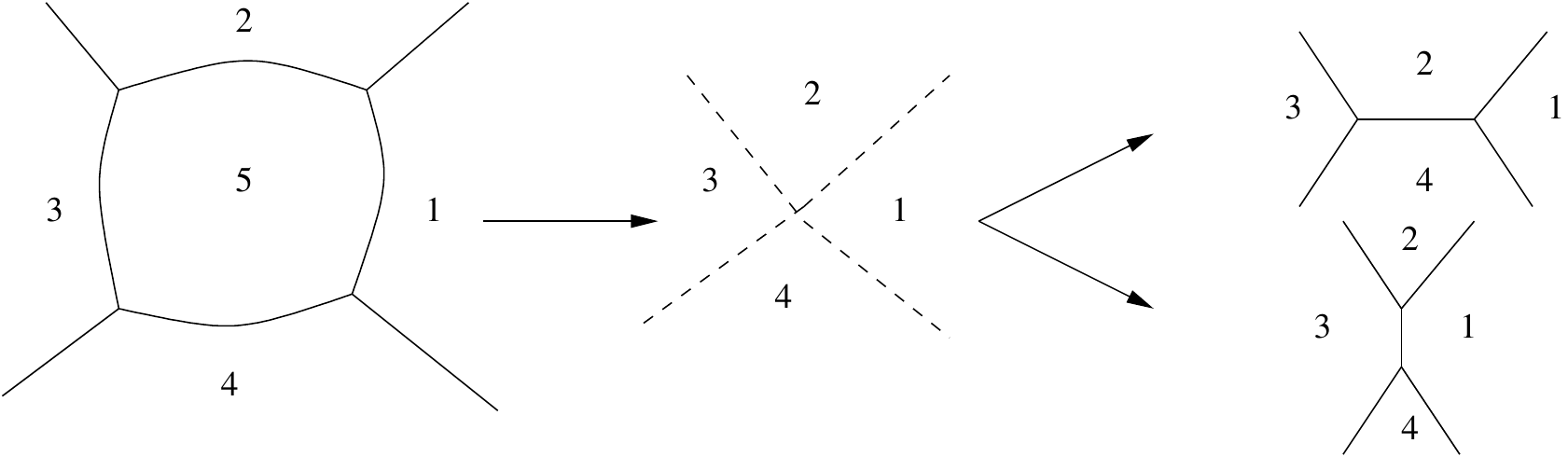}%
\end{minipage}%
\hspace{0.1\textwidth}%
\begin{minipage}[c]{0.4\textwidth}%
\includegraphics[width=\textwidth]{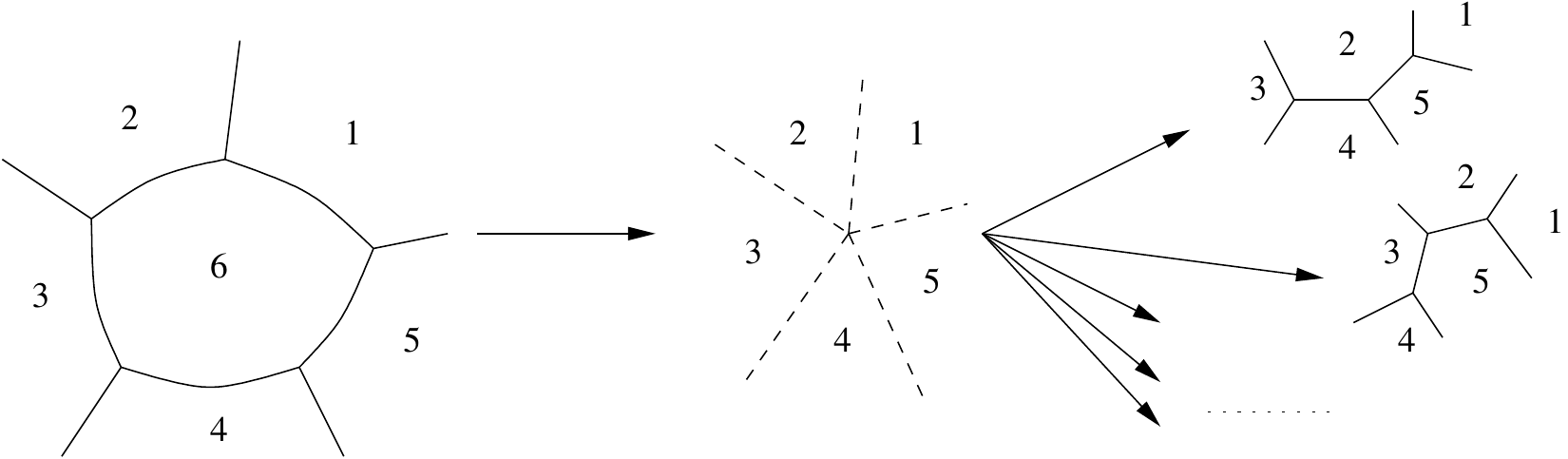}%
\end{minipage}%
}%
\caption{Grain vanishing} \label{fig:vanishing}
\end{figure}
Grains with $n=2$ and $n=3$ vanish in a single possible way. For
$n=4$ we observe two topologically distinguishable possibilities and
for $n=5$ even five possible local configurations (Fig.
\ref{fig:vanishing}). For further details on the resulting
topological classes we refer to the review article by Fradkov and
Udler \cite{FU94}.\par It is unclear by which mechanism a specific
topological configuration is selected within switching or after
vanishing events. A natural idea is to compute all possible local
configurations and select the one that minimises energy locally in
the best way, see \cite{HNO03}.
%
%
\subsection{Kinetic model for networks of grains}
%
Our next aim is to derive the kinetic model for large networks of grains with triple junctions. This
kinetic model comprises the same essential features as the gradient flow dynamics, but differs 
in some aspects. 
%
\paragraph*{One--particle distribution}
%
Following \cite{Fr88a,FU94} we introduce a
number density ${f_{n}\left(a,t\right)}$ that measures the number of
grains with topological class $n\geq2$ and area $a\geq0$ at time
$t\geq0$. Using the {\it von Neumann--Mullins law} \eqref{vNMl} we
can describe the evolution of $f$ by transport equations
\begin{align}
\label{transportsystem}
\partial_{t} {f_{n}\left(a,t\right)} + \left(n-6\right) \partial_{a} {f_{n}\left(a,t\right)} &= 0
\end{align}
as long as no topological rearrangements take place. Furthermore we
choose the following boundary conditions
\begin{align*}
{f_{n}\left(0,t\right)} = 0 \quad\quad\text{for}\quad  n > 6,
\end{align*}
ensuring that no additional mass is transported from the negative
half--axis to the positive one. This means no additional grains can
be created.
%
%
\paragraph*{`Collision' operator}
%
%
To model topological changes we introduce a collision term
$(\tilde{J}f)_{n}$ on the r.h.s. of \eqref{transportsystem} which couples
the equations for different topological classes. We define {\it
topological fluxes} $\eta_{n}^{+}$ and $\eta_{n}^{-}$ denoting the
flux from class $n$ to $n+1$ and from $n$ to $n-1$, respectively, so
that
\begin{equation*}
\nat{\tilde{J}f}_{n} = \eta_{n-1}^{+} + \eta_{n+1}^{-} -
\eta_{n}^{+} - \eta_{n}^{-}\,.
\end{equation*}
Next we state the so--called {\it `gas' approximation} of
the collision term by Fradkov \cite{Fr88a} that takes into account the following transitions between 
topological classes.

\begin{itemize}
\item
Switching events 
cause simultaneously both a transition from $n$ to $\left(n-1\right)$ 
(grains that contained the vanishing edge) and a transition from $n$ to $\left(n+1\right)$ (grains that contain the created edge).
\item
Vanishing of a neighbouring grain corresponds to a transition from
$n$ to $\left(n-1\right)$.
\end{itemize}
The restriction to these elementary events for topology changes is 
a simplification of the real dynamics in networks of grains. 
In particular, we ignore that
the topological class of a grain is lowered by
two if the neighbouring annihilated grain was a lens with
topological class $n=2$. Fradkov and Udler argue \cite{FU94} that
such an event takes place only very rarely as the number of lenses
itself is already very small. In addition, we ignore that the annihilation of a grain with topology class $n=5$ causes
other grains to increase their topology class, compare Figure \ref{fig:vanishing}.
%
%
%
%
\paragraph*{Transition rates}
%
We make a further simplification by assuming constitutive relations for the transitions 
rates, that in particular ignore all neighbour correlations. More precisely, in what follows we assume that the topological fluxes are given by
\begin{equation}
\label{intro.Fluxes}
\eta_{n}^{+} = \Gamma\beta\, n f_{n}\,,\quad \eta_{n}^{-} =
\Gamma\left(\beta+1\right) n f_{n},
\end{equation}
where the coupling weight $\Gamma$ 
describes the
intensity of topological changes and depends in a
self--consistent way on the complete state of the system, see \eqref{intro.DefGa} below. The parameter $\beta$ in \eqref{intro.Fluxes} measures the ratio between switching events and vanishing events.
In principle one could allow for arbitrary values of $\beta$
but our analysis requires $0<\beta<2$. Moreover, \cite{Fr88b} provide numerical evidence that
experimental data correspond to $0.45\lesssim\frac{\beta}{1+\beta}\lesssim0.6$, that means
$0.8\lesssim\beta\lesssim1.5$.
\par%
According to \eqref{intro.Fluxes} the collision terms are given by
$\tilde{J}f=\Gamma\at{f}Jf$
with 
\begin{align}
\label{intro.DefJ}
\begin{split}
\at{Jf}_{2} &= 3\left(\beta+1\right)f_{3}-2\beta f_{2},
\\%
\at{Jf}_{n} &= \left(\beta+1\right)\left(n+1\right)f_{n+1}-
\left(2\beta+1\right)n\,f_{n}+\beta\left(n-1\right)f_{n-1}
\quad\text{for}\quad{}n>2.
\end{split}
\end{align}
Notice that this definition ensures the {\it zero balance property},
that is
\begin{align}
\label{zerobalance} \sum_{n\ge2}\left(Jf\right)_{n}\pair{a}{t} = 0\qquad\text{for all}\;a,\,t\geq0.
\end{align}
This identity reflects that for each $a_0>0$ the number of grains with area $a_0$ does not change
due to neither switching nor vanishing events.
%
\paragraph*{Evolution equation}

The coupling weight $\Gamma$ which makes the equations nonlinear
(and nonlocal in the grain area variable $a$) is chosen as
\begin{align*}
\Gamma\at{f}= \Gamma_N\at{f}/\Gamma_D\at{f}
\end{align*}
with
\begin{align}
\label{intro.DefGa}  %
\Gamma_N\at{f\at{t}}=\sum_{n\ge2} \left(n-6\right)^2
{f_{n}\left(0,t\right)},
\quad%
\Gamma_D\at{f\at{t}}={\sum_{n\ge2}n\int\limits_{0}^{\infty}{f_{n}
\left(a,t\right)}\,\dint{a} -
2\left(\beta+1\right)\int\limits_{0}^{\infty}{f_{2}\left(a,t\right)}
\,\dint{a}}.
\end{align}
We will see in \S\ref{sec:Model.Props} below, that this choice of
$\Gamma$ guarantees consistency as the solution satisfies the
polyhedral formula, compare \eqref{intro.PolyhedralFormula}, and
conserves the total area covered by the grains under the evolution.
\bigpar
The kinetic model we consider in this paper is thus given by
\begin{align}
\label{infinitesystem}
\partial_{t} {f_{n}\left(a,t\right)} +
 \left(n-6\right) \partial_{a} {f_{n}\left(a,t\right)} &=
\Gamma\left(f\left(t\right)\right) {\left(J
f\right)_{n}\left(a,t\right)}.
\end{align}
These equations
\eqref{infinitesystem} are basically the same as in the work of
Fradkov \cite{Fr88a,Fr88b,FU94}. The coupling term
$\left(Jf\right)_{2}$ differs and we do not neglect $\int\limits
f_{2}\,\dint{a}$ within $\Gamma_D\left(f\right)$. We refer moreover 
to \cite{BKLT06}, which presents some formal analysis and numerical simulations for
a similar kinetic model. This model relies on different
expressions for $\eta^\pm_n$ but comprises the same essential features as \eqref{infinitesystem}.
\bigpar
Finally, we mention the following implication of Definition \eqref{intro.DefGa}.
As long as no vanishing events take place, that means as long as no small grains with topology 
class $2\leq{n}\leq5$ do exist, switching events do likewise not occur, and \eqref{infinitesystem} reduces to a system of uncoupled transport equations. This is in contrast to the gradient flow dynamics in which 
edge switching events occur independently of grain vanishing.
%
\subsection{Qualitative properties of the kinetic model}\label{sec:Model.Props}
%
%
We summarise the most important properties of the kinetic model. In
order to simplify the presentation we argue by means of formal
analysis but mention that all results will be proven rigorously
within \S\ref{sec:Proofs}.
%
%
\paragraph*{Decreasing number of grains}
%
Since equations \eqref{infinitesystem} reflect a coarsening process,
it is clear that the total number of grains $N\at{f\at{t}}$
with
\begin{align}
\label{intro.DefN}
 N\at{f} =
\sum_{n\ge2}\int\limits_{0}^{\infty}{f_{n}\at{a}}\,\dint{a}
\end{align}
should decrease in time. This is also satisfied for solutions to our
model as grains with topology class $n<6$ can shrink to area zero so that they are annihilated. 
More precisely, the evolution equations imply 
\begin{equation*}
\notag
\frac{\dint{}}{\dint{t}} N\at{f\at{t}} =
\sum_{n\ge2}\left(n-6\right){f_{n}\left(0,t\right)}
=\sum_{n=2}^5\left(n-6\right){f_{n}\left(0,t\right)}
\leq 0\,.
\end{equation*}
%
%
\paragraph*{Polyhedral formula}
%
In this section we motivate the choice of $\Gamma$ in
\eqref{intro.DefGa}.
It is essential that our kinetic model reflects all the properties
which are satisfied by a grain configuration which covers the
complete area and where edges only meet in triple junctions. Hence
we need to ensure that  Euler's polyhedral formula is satisfied.
For a finite network of grains Poincar{\'e}'s version of the polyhedral formula reads
\begin{equation*}
V + F - E = \chi\at{g}.
\end{equation*}
with $V$, $E$, and $F$ being the number of vertices, edges, and
facets, respectively. Moreover, $g$ is the genus of the surface and
$ \chi\left(g\right) = 2 - 2g $
the corresponding Euler characteristic. We can encode that grain
boundaries only meet in triple junctions by setting $V=2/3 E$, and
the polyhedral formula reduces to
\begin{equation*}
3F - E = 3\chi\at{g}.
\end{equation*}
In what follows we set $\chi\at{g}=0$ because 
the network of grains is usually considered on a two-dimensional torus. 
In the kinetic model the \emph{normalised} numbers of facets and edges are
given by
\begin{align*}
F\at{f} =N\at{f}=
\sum_{n\ge2}\int\limits_{0}^{\infty}{f_{n}\at{a}}\,\dint{a},\quad
E\at{f} = \frac{1}{2}\sum_{n\ge2}
n\int\limits_{0}^{\infty}{f_{n}\at{a}}\,\dint{a},
\end{align*}
respectively, and consequently we require
each solution of \eqref{infinitesystem} to satisfy
\begin{align}
\label{intro.PolyhedralFormula}
P\at{f\at{t}}=0
\end{align}
for all times $t\geq0$, where the \emph{polyhedral defect} is
given by
\begin{align}
\label{intro.DefP}
 P\at{f} =
\sum_{n\ge2}\at{n-6}\int\limits_{0}^{\infty}{f_{n}\at{a}}\,\dint{a}.
\end{align}
The main observation is that our choice of $\Gamma$ guarantees \eqref{intro.PolyhedralFormula}. Indeed, due to \eqref{zerobalance} we find
\begin{equation*}
\begin{split}
\frac{\dint{}}{\dint{t}}P\at{f\at{t}}=
&-\sum_{n\ge2}\left(n-6\right)^2\int\limits_{0}^{\infty}\partial_{a}{f_{n}
\left(a,t\right)}\,\dint{a} +
\sum_{n\ge2}\left(n-6\right)\int\limits_{0}^{\infty}\Gamma\left(f
\left(t\right)\right){\left(J f\right)_{n}\left(a,t\right)}\,\dint{a}\\
= &\sum_{n\ge2}\left(n-6\right)^2 f_n\left(0,t\right) +
\Gamma\left(f\left(t\right)\right)\sum_{n\ge2}
n\int\limits_{0}^{\infty}{\left(J f\right)_{n}
\left(a,t\right)}\,\dint{a},
\end{split}
\end{equation*}
and a simple
calculation shows
\begin{equation*}
\begin{split}
\sum_{n\ge2}n\left(Jf\right)_n &=
\beta\sum_{n\ge3}n\left(n-1\right)f_{n-1} +
\left(\beta+1\right)\sum_{n\ge2}n\left(n+1\right)f_{n+1} -
\beta\sum_{n\ge2}n^2f_{n} - \left(\beta+1\right)\sum_{n\ge3}n^2f_{n}\\
&= 2\left(\beta+1\right)f_{2}-\sum_{n\ge2}nf_{n} ,
\end{split}
\end{equation*}
which implies
\begin{equation*}
\frac{\dint{}}{\dint{t}}P\at{f\at{t}}=\Gamma_N\at{f}
-\Gamma\at{f}\Gamma_D\at{f}=0
\end{equation*}
thanks to \eqref{intro.DefGa}. Hence, the polyhedral formula
\eqref{intro.PolyhedralFormula} is satisfied for all $t>0$ if it is
satisfied by the initial data.
%
\paragraph*{Conservation of area}
%
As a consequence of the polyhedral formula we obtain that the total
covered area $A\at{f\at{t}}$ with
\begin{align}
\label{intro.DefA}
 A\at{f} =
\sum_{n\ge2}\int\limits_{0}^{\infty}a{f_{n}\at{a}}\,\dint{a}
\end{align}
is a conserved quantity. This follows from
\begin{equation*}
\begin{split}
\frac{\dint{}}{\dint{t}}A\at{f\at{t}} &=
-\sum_{n\ge2}\left(n-6\right)\int\limits_{0}^{\infty}
a\partial_{a}{f_{n}\left(a,t\right)}\,\dint{a}
 + \Gamma\left(f\left(t\right)\right)\int\limits_{0}^{\infty} a
\sum_{n\ge2}{\left(J f\right)_{n}\left(a,t\right)}\,\dint{a}\\
&=
\sum_{n\ge2}\left(n-6\right)
\int\limits_{0}^{\infty}{f_{n}\left(a,t\right)}\,\dint{a}=P\at{f\at{t}}=0,
\end{split}
\end{equation*}
where we used an integration by parts and the zero balance property
\eqref{zerobalance}.
%
\subsection{Statement of the main result}
%
%
Our main result in this paper concerns the existence of mild solutions to the following initial and boundary value problem, provided that the initial data satisfy certain assumptions.
\begin{problem}
\label{TheInfiniteProblem}
For fixed $0<T<\infty$, and given initial data $g=g_n\at{a}$ we seek
mild solutions $f=f_n\pair{a}{t}$ to the following infinite system of coupled transport equations
\begin{align*}
\partial_{t} {f_{n}\left(a,t\right)} + \left(n-6\right)
\partial_{a} {f_{n}\left(a,t\right)} &= \Gamma\left(f\left(t\right)\right)
{\left(J f\right)_{n}\left(a,t\right)}
\end{align*}
with initial and boundary conditions
\begin{align*}
\begin{array}{lclclc}
f_n\pair{a}{0} &=& g_n\at{a}&\quad&\text{for}& n\geq2,
\\%
f_n\pair{0}{t}&=&0&\quad&\text{for}& n>6,
\end{array}
\end{align*}
where $a\geq0$, $t\in\ccinterval{0}{T}$, $n\geq2$.
\end{problem}
A first assumption we have to make concerns the choice of $\beta$.
In what follows we always suppose that $\beta$ is a fixed constant
with $0<\beta<2$, where the upper bound is necessary in order to
ensure that $\Gamma_D\at{f}$ is non-negative for all $f\geq0$.
Further necessary assumptions regard the initial data.
\begin{assumption}
\label{TheMainAssumption}%
Suppose that
\begin{enumerate}
\item $g$ is non--negative with $g_n\at{0}=0$ for all $n>6$,
\item $\sum\limits_{n\geq2}\at{1+n}\norm{g_n}_{0,\,1}<\infty$ with $\norm{g_n}_{0,\,1}=\int\limits_0^\infty\at{1+a}g_n\at{a}\dint{a}$,
\item $g$ fulfils the polyhedral formula $P\at{g}=0$.
\end{enumerate}
Moreover, suppose that $g$ is sufficiently regular and decays sufficiently fast in $n$- and $a$-direction,
\end{assumption}
According to the discussion in \S\ref{sec:model}, the first three assumptions on the initial data appear very naturally. Our regularity and decay assumptions, however, are needed for technical reasons and can probably be weakened at the price of more analytical effort.  The precise statement of these assumptions appears below, but we mention that we mainly assume all functions $g_n$ to be equi-continuous and to decay
exponentially with respect to $a$ and $n$.
\par%
Our main result can be summarised as follows.
\begin{theorem}
\label{TheMainTheorem}
For any initial data that satisfy Assumption \ref{TheMainAssumption}
there exists a unique mild solution $f$ to Problem
\eqref{TheInfiniteProblem} for all $0\leq{t}<\infty$. Moreover, this
solution conserves the area with non--increasing number of grains,
and all states $f\at{t}$ satisfy Assumption \ref{TheMainAssumption}.
\end{theorem}
The details of the proof are presented within \S\ref{sec:Proofs} and
rely on the following key ideas. In \S\ref{sec:Proofs.1} we
introduce an approximate system of coupled transport equations by
neglecting all topological classes with $n>n_{0}$.
\par
In order to construct mild solutions for the approximate system we solve the
transport equations explicitly, and apply \emph{Duhamel's Principle}
(or \emph{Variation of Constants}) to the collision operator. This
is discussed in \S\ref{sec:Proofs.Transport}.
\par
Freezing the coupling weight $\Gamma$ in the approximate equations
we can derive a comparison principle that implies both, the
non-negativity of solutions and the existence of an appropriate
super--solution. This will be done in \S\ref{sec:AuxProb}.
\par
In \S\ref{sec:Proofs.LocalSol} we establish the short time existence
and uniqueness for admissible solutions to the approximate system.
To this end we propose a suitable iteration scheme, and make use of
Banach's Contraction Principle.
\par
The decay behaviour with respect to $a$ and $n$ is investigated in \S\ref{sec:Proofs.Tightness}. We derive several tightness estimates, and as a consequence we obtain long-time existence for the approximate system together with estimates that are uniform in $n_0$.
\par
In \S\ref{sec:Proofs.ApproxLimit} we show the existence of mild solutions to Problem \ref{TheInfiniteProblem} by passing to the limit $n_0\to\infty$. The Arzela--Ascoli theorem provides that solutions to the approximate problem
converge to some reasonable limit, and the tightness estimates ensure that this limit provides an admissible solution.
Finally, we sketch how both, the uniqueness of solutions and the continuous dependence
on the initial data, can be obtained.
%
%
\subsection{Further considerations}
%
%
Here we point out some open questions and directions for future
research.
%
%
\paragraph*{Stationary solutions}
%
Nontrivial stationary solutions to \eqref{infinitesystem}, i.e. $f^s
\not\equiv 0$, are characterised by
$ f_n^s\left(a\right) = 0 $
for $n \neq 6$ and $0<a<\infty$ (cf. \cite{Hen07}, \S6.1), but the
component function $f_6^s\left(a\right)$ can be anything. 
This is a further difference to the gradient flow dynamics 
as there not all networks of hexagons are stationary. However,
such nontrivial stationary solutions to the kinetic model are expected to be unstable for
the following reasons: Slightly perturbed data lead to a positive
$\Gamma\left(f\left(t\right)\right)$ for some times $t$ and are
therefore affected by the coupling operator ${\left(J
f\right)_{n}\left(a,t\right)}$. This leads to a decrease of the
total number of grains. We therefore expect stationary solutions to
be unstable, and thus we suspect that $f\rightarrow0$ weakly for
$t\rightarrow\infty$.
%
%
\paragraph*{Self--similar scaling}
%
For a coarsening process as considered here, one usually expects to
find self--similarity under dynamic scaling. The natural rescaling
${\varphi_{n}\left(\xi,t\right)}=t^{2}{f_{n}\left(a,t\right)}$,
$\xi=a/t$,
yields the following equation for self--similar solutions
\begin{align}
\label{rescaled:simple:stationary}
\left(n-6-\xi\right)\partial_{\xi}\varphi_{n}=
\Gamma\left(\varphi\right)\left(J\varphi\right)_{n}+2\varphi_{n}\,,\quad
n\ge2\,,\quad\xi\ge0\,.
\end{align}
The natural boundary conditions are
$\varphi_{n}\left(0\right)=0 $
for $n>6$, so that the solution depends on the values of
$\varphi_{2}\left(0\right),\dots,\varphi_{5}\left(0\right)$. Note
that the coupling weight still depends on the complete solution. A
starting point for future analysis is the following observation. We
can integrate \eqref{rescaled:simple:stationary} with respect to
$\xi$ to obtain
\begin{align}
\notag
\left(6-n\right)\varphi_{n}\left(0\right)=
\Gamma\left(\varphi\right)\left(J\phi\right)_{n}+\phi_{n}\,,\quad
n\ge2.
\end{align}
This is a two--point iteration scheme for 
$\phi_n=\int\limits\varphi_{n}\left(\xi\right)\dint{\xi}$.
%
%
\paragraph*{Lewis' law}
%
A natural question concerning grain growth is to ask whether there
are correlations between the topological class and the area of a
grain. Lewis \cite{L43} observed a linear relationship examining
cellular structures arising in biology, and Rivier and
Lissowski \cite{RL82} derived {\it Lewis' law} by maximum entropy
arguments applied to cell distributions. In common with
Flyvbjerg \cite{F93} this so--called {\it Lewis' law} reads
\begin{align}
\label{lewislaw} \left<\xi\right>_{n}=b\left(n-6\right)+c
\end{align}
in our model. Here
$\left<\xi\right>_{n}=\int\limits\xi\varphi_{n}\left(\xi\right)\dint{\xi}
/ \int\limits\varphi_{n}\left(\xi\right)\dint{\xi}$ denotes the mean grain
size in the topological class $n$. However, it is unclear if this
phenomenological law is really applicable for grain growth.
%
%
Formal computations suggest in our model that \eqref{lewislaw} is
valid asymptotically for large $n$ with $b=1/\left(\Gamma+1\right)$
and $c=b\left(\left(2\beta+1\right)-6\Gamma\right)$
(cf. \cite{Hen07}, \S6.4). Similar results are achieved by
Flyvbjerg \cite{F93}.

\section{Proof of the main result}
\label{sec:Proofs}
%
%
\subsection{The approximate system}
\label{sec:Proofs.1}
%
%
The approximate system for \eqref{infinitesystem} results from the
original equations by neglecting all number densities belonging to
topological classes with $n>{n_0}$. More precisely, we choose the parameter $n_0$
with $6<n_{0}<\infty$, and modify the coupling operator accordingly. 
The approximate coupling operator $J$ splits into its gain and loss part, that is
\begin{align}
\label{ApproxProb.CouplingOperator.Eqn1} J=J_+-J_-,
\end{align}
which now are given by
\begin{align}
\label{ApproxProb.CouplingOperator.Gain}
\begin{split}
\left(J_+f\right)_{2}&=3\left(\beta+1\right)f_{3},\\
\left(J_+f\right)_{n}&=\left(\beta+1\right)\left(n+1\right)f_{n+1}+
\beta\left(n-1\right)f_{n-1}\quad\text{for
$2<n<n_0$,}\\
\left(J_+f\right)_{n_{0}}&=\beta\left(n_{0}-1\right)f_{n_{0}-1},
\end{split}
\end{align}%
and
\begin{align}
\label{ApproxProb.CouplingOperator.Loss}
\begin{split}
\left(J_-f\right)_{2}&=2\beta{}f_{2},\\
\left(J_-f\right)_{n}&=\left(2\beta+1\right)nf_{n}\quad\text{for $2<n<n_0$,}\\
\left(J_-f\right)_{n_{0}}&=\left(\beta+1\right)n_{0}f_{n_{0}}.
\end{split}
\end{align}%
Notice that for $n<n_0$ the term $\at{Jf}_n$ is defined as in the
original model. For the sake of consistency we must moreover adapt
the formula for $\Gamma$. In what follows we use the approximate
coupling weight
\begin{align}
\label{ApproxProb.CouplingWeights}%
\Gamma\at{f}=\Gamma_N\at{f}/\Gamma_D\at{f},\quad
\Gamma_N\at{f}=\sum\limits_{n=2}^{5}\at{n-6}^2f_n\at{0},\quad
\Gamma_D\at{f}=-\sum\limits_{n=2}^{n_0}
\int\limits_0^\infty{}n\at{Jf}_n\at{a}\,\dint{a}.
\end{align}
Analogously to \S\ref{sec:model} we define the \emph{total
area} $A\at{f}$, the \emph{number of grains} $N\at{f}$, and the
\emph{polyhedral defect} $P\at{f}$ by
\begin{align*}
N\at{f}=\sum\limits_{n=2}^{n_0}\int\limits_{0}^{\infty}{f}_n\at{a}\,\dint{a},\quad
A\at{f}=\sum\limits_{n=2}^{n_0}\int\limits_{0}^{\infty}a{f}_n\at{a}\,\dint{a},\quad
P\at{f}=\sum\limits_{n=2}^{n_0}\at{n-6}\int\limits_{0}^{\infty}{f}_n\at{a}\,\dint{a},
\end{align*}
and for convenience we introduce in addition
\begin{align*}
M\at{f}=\sum\limits_{n=2}^{n_0}n\int\limits_{0}^{\infty}{f}_n\at{a}\,\dint{a},\quad
R\at{f}=2\at{\beta+1}\int\limits_{0}^{\infty}{f}_2\at{a}\,\dint{a}-n_0\beta\int
\limits_{0}^{\infty}{f}_{n_0}\at{a}\,\dint{a}.
\end{align*}
\begin{remark}
\label{AppSys.J.Props}
\begin{enumerate}
\item
Definitions \eqref{ApproxProb.CouplingOperator.Gain} and
\eqref{ApproxProb.CouplingOperator.Loss} imply
\begin{align}
\label{ApproxProb.BasicProps.Eqn1}
\sum\limits_{n=2}^{n_{0}}\at{Jf}_n=0,\quad%
\sum\limits_{n=2}^{n_{0}}n\at{Jf}_n=
2\at{\beta+1}f_2-n_0\beta{f}_{n_0} -
\sum\limits_{n=2}^{n_{0}}nf_{n},
\end{align}
and we infer that
\begin{align}
\label{ApproxProb.BasicProps.Eqn2} %
A\at{Jf}=N\at{Jf}=0,\qquad
P\at{Jf}=M\at{Jf}=-\Gamma_D\at{f}.
\end{align}
\item
The polyhedral formula
$P\at{f}=0$ implies
\begin{align*}
\Gamma_D\at{f}=-R\at{f}+6N\at{f},
\end{align*}
and if $f$ is in addition non--negative we have, thanks to
$0<\beta<2$,
\begin{align}
\label{ApproxProb.BasicProps.Eqn3}
\Gamma_D\at{f}\geq\at{6-2\at{\beta+1}}N\at{f}>0.
\end{align}
\end{enumerate}
\end{remark}
The approximate problem we aim to solve can now be stated as follows.
\begin{problem}
\label{ApproxProb}%
For fixed $n_0>6$ and $0<T<\infty$, and given initial data $g=g_n\at{a}$ we seek
(mild) solutions $f=f_n\pair{a}{t}$ to
\begin{align}
\label{ApproxProb.EvolEqn}%
\partial_{t} {f_{n}\left(a,t\right)} + \left(n-6\right)
\partial_{a} {f_{n}\left(a,t\right)} &= \Gamma\left(f\left(t\right)\right)
{\left(J f\right)_{n}\left(a,t\right)}
\end{align}
with initial and boundary conditions
\begin{align}
\label{ApproxProb.ICandBC}
\begin{array}{lclclc}
f_n\pair{a}{0} &=& g_n\at{a}&\quad&\text{for}& 2\leq{n}\leq{n_0},
\\%
f_n\pair{0}{t}&=&0&\quad&\text{for}& 7\leq{n}\leq{n_0},
\end{array}
\end{align}
where $a\geq0$, $t\in\ccinterval{0}{T}$, $2\leq{n}\leq{n_0}$, and
$J$ and $\Gamma$ are defined as in
\eqref{ApproxProb.CouplingOperator.Gain}--\eqref{ApproxProb.CouplingWeights}.
\end{problem}
%
%
%
%
%
%
%
\subsection{Transport equations and notion of mild solutions}
\label{sec:Proofs.Transport}
%
%
%
%
%
%
We start with some basic facts about solutions to transport
equations in the upper-right space-time quadrant $a\geq0$ and
$t\geq0$. For this reason let us consider the following system of
transport equations.
\begin{problem}
\label{LinProb1}%
Let $a\geq0$, $2\leq{n}\leq{n_0}$ and $t\in\ccinterval{0}{T}$. For
fixed initial data $g=g_n\at{a}$ and given right hand side
$h=h_n\pair{a}{t}$ we seek (mild) solutions $f=f_n\pair{a}{t}$ to
\begin{align}
\label{LinProb1.Eqn}
\partial_tf_n+\at{n-6}\partial_af_n&=h_n
\end{align}
with initial and boundary conditions as in the approximate problem,
see \eqref{ApproxProb.ICandBC}.
\end{problem}
The \emph{homogeneous} problem with $h=0$ can be solved explicitly
by the \emph{method of characteristics}.  This means, the general
solution to the homogeneous problem is given by
$f\at{t}=\calT\at{t}{g}$, where the \emph{group of transport
operators} is defined by
\begin{align*}
\at{\calT\at{t}g}_n=\calT_{n-6}\at{t}g_n
\end{align*}
with
\begin{align}
\notag
\nat{\calT_{n-6}\at{t}g_n}\at{a}= g_n\at{a-\at{n-6}t}
\end{align}
for all $n\leq{6}$, whereas $n>6$ corresponds to
\begin{align}
\notag
\nat{\calT_{n-6}\at{t}g_n}\at{a}=\left\{\begin{array}{lcl}
g_n\at{a-\at{n-6}t}&\text{for}& a \geq{\at{n-6}t}, %
\\%
0&\text{for}& a < \at{n-6}t.%
\end{array}\right.
\end{align}
Recall that $n\leq6$ implies the transport velocity $n-6$ to be
non--positive, so there is no contribution from the boundary in this
case.
\bigpar%
For non--vanishing right hand side $h$ the solutions  can be
constructed by means of \emph{Duhamel's Principle}. More precisely,
for given $h=h_n\pair{a}{t}$ the unique mild solution to Problem
\ref{LinProb1} is given by
\begin{align}
\label{LinProb1.MildSol}
f\at{t}=\calT\at{t}{g}+\int\limits_0^t\calT\at{t-s}h\at{s}\,\dint{s}.
\end{align}
Moreover, the mild solution depends continuously on the data via
\begin{align}
\label{LinProb1.ExMildSol.Eqn1} \norm{f\at{t}}_\infty
\leq%
\norm{g}_\infty+\int\limits_0^t\norm{h\at{s}}_\infty\,\dint{s},
\quad\quad%
\norm{f\at{t}}_{0,\,1}
\leq%
C_{n_0,\,t}\Bat{\norm{g}_{0,\,1}+\int\limits_0^t\norm{h\at{s}}_{0,\,1}\,\dint{s}},
\end{align}
where
\begin{math}
\norm{g}_{0,\,1}:=\sum_{n=2}^{n_0}\int_0^\infty\at{1+a}\abs{g_n\at{a}}\dint{a}<\infty.
\end{math}
\bigpar%
In the sequel we make use of the following regularity results for
mild solutions, that can be derived directly from the representation
formula given in \eqref{LinProb1.MildSol}.
\begin{definition}
\begin{enumerate}
\item
A state $g=g_n\at{a}$ is called
\begin{enumerate}
\item
\emph{$C^0$-regular}, if $g$ is continuous (w.r.t. to $a$) and
 $0=g_n\at{0}$ for all $7\leq{n}\leq{n_0}$,
\item
\emph{$L^1$-regular}, if $\norm{g}_{0,\,1}<\infty$.
\end{enumerate}
\item
The right hand side $h=h_n\pair{a}{t}$ is called
\begin{enumerate}
\item
\emph{$C^0$-regular}, if $h$ is continuous (w.r.t. to
$\pair{a}{t}$),
\item
\emph{$L^1$-regular}, if
$\int_0^{T}\norm{h\at{t}}_{0,\,1}\,\dint{t}<\infty$.
\end{enumerate}
\end{enumerate}
Moreover, we say the solution $f=f_n\pair{a}{t}$ is \emph{regular},
if $f$ is a regular right hand side and $f\at{t}$ is a regular state
for all $t\in\ccinterval{0}{T}$, where `regular' means either $C^0$-
or $L^1$-regular.
\end{definition}
\begin{lemma}
\label{LinProb1.PropMildSol}%
The mild solution $f$ to Problem \ref{LinProb1} given in
\eqref{LinProb1.MildSol} has the following properties.
\begin{enumerate}
\item
If the data are $C^0$-regular, then $f$ is $C^0$-regular. In
particular we have
\begin{math}
f_n\pair{0}{t}=0
\end{math}
for all $t$ and $7\leq{n}\leq{n_0}$.
\item%
If the data are $L^1$-regular, then $f$ is $L^1$-regular.
\end{enumerate}
\end{lemma}%
Finally, we collect some properties of mild solutions to be used in
\S\ref{sec:Proofs.LocalSol}.
\begin{remark}
If the data $g$ and $h$ are compactly supported in
$\{a:0\leq{a}\leq{a_0}\}$, then the mild solution $f\at{t}$ is
compactly supported in $\{a:0\leq{a}\leq{a_0}+\at{n_0-6}t\}$
\end{remark}
%
%
%
%
\begin{lemma}
\label{LinProb1.Moments}%
Suppose that the data $\pair{g}{h}$ for Problem \ref{LinProb1} are
$C^0$- and $L^1$-regular. Then the mild solution $f$ satisfies
\begin{align}
\label{LinProb1.Moments.Eqn1}%
\begin{split}
N\at{f\at{t}}&=N\at{g}+\int\limits_{0}^tN\at{h\at{s}}\,\dint{s}+
\sum_{n=2}^{5}\at{n-6}\int\limits_{0}^t{f_n}\pair{0}{s}\,\dint{s},
\\%
P\at{f\at{t}}&=P\at{g}+\int\limits_{0}^tP\at{h\at{s}}\,\dint{s}
+\int\limits_{0}^t\Gamma_D\at{f\at{s}}\,\dint{s},
\\%
A\at{f\at{t}}&=A\at{g}+\int\limits_{0}^tA\at{h\at{s}}\,\dint{s}+
\int\limits_{0}^tP\at{f\at{s}}\,\dint{s}.
\end{split}
\end{align}
In particular, $N\at{f\at{t}}$, $P\at{f\at{t}}$, and $A\at{f\at{t}}$ are continuously 
differentiable with respect to $t$.
\end{lemma}
\begin{proof}
For classical solutions the differential counterparts of all
assertion follow directly from the differential equation
\eqref{LinProb1.Eqn}. 
For mild solutions we approximate with
classical solutions. In fact, we can approximate the data $\npair{g}{h}$ by some regular data $\npair{\tilde{g}}{\tilde{h}}$ that satisfy the boundary condition and are both differentiable and $L^1$-regular. The corresponding solution $\tilde{f}$ to Problem \ref{LinProb1} is then a classical solution and satisfies \eqref{LinProb1.Moments.Eqn1}. 
Finally, $\tilde{f}-f$ is a mild solution to Problem \ref{LinProb1}
with data $\npair{\tilde{g}-g}{\tilde{h}-h}$, and using \eqref{LinProb1.ExMildSol.Eqn1} we infer that $f$ satisfies \eqref{LinProb1.Moments.Eqn1}.
\end{proof}
%
%
%
\subsection{Auxiliary problem with prescribed coupling weight $\Gamma$}
\label{sec:AuxProb}
%
%
In this section we consider an auxiliary problem that results from
the approximate system \ref{ApproxProb} by prescribing the coupling
weight $\Gamma$ as a function of time, and study both the existence
and qualitative properties of solutions in the space of bounded and
continuous functions. In particular, we will derive comparison results
for such solutions.
\par
Below in \S\ref{sec:Proofs.LocalSol} we apply
our results in the context of the approximate problem, and show
that each solutions to the approximate problem must be non--negative
and bounded from above by some appropriately chosen super--solution.
\bigpar%
The auxiliary problem can be stated as follows.
\begin{problem}
\label{AuxProb.Prob}
Let $0<T<\infty$ be arbitrary, $\Gamma\in{C}\at{\ccinterval{0}{T}}$
be some non--negative weight function, and $g$ be some initial data.
Then we seek (mild) solutions $f$ to
\begin{align}
\label{AuxProb.Form.1}
\partial_{t} {f_{n}} + \left(n-6\right)
\partial_{a} {f_{n}}
&= \Gamma\at{t} {\left(Jf\right)_{n}}
\end{align}
with initial and boundary conditions \eqref{ApproxProb.ICandBC}.
\end{problem}
The state space and solution space for Problem \ref{AuxProb.Prob}
are given by
\begin{align*}
\sspAux=BC\at{\cointerval{0}{\infty};\,\xspAux}\quad\text{and}\quad
\fspAux=C\at{\ccinterval{0}{T};\sspAux},
\end{align*}
respectively, where $\xspAux\cong\Rset^{n_0-1}$ denotes the set of
all tuples $\triple{x_2}{...}{x_{n_0}}$.
\bigpar%
The key idea for arriving at comparison results for the auxiliary
problem is to split the exchange operator into its loss and gain
part according to \eqref{ApproxProb.CouplingOperator.Eqn1}. More
precisely, we regard each solution to the approximate problem
\eqref{AuxProb.Form.1} as mild solution to
\begin{align}
\label{AuxProb.Form.2}
\partial_{t} {f_{n}} + \left(n-6\right)
\partial_{a} {f_{n}} + \Gamma\at{t}J_-f\at{t}=h\at{t},
\end{align}
where the right hand side is given by
\begin{align}
\label{AuxProb.Form.3} h\at{t}=\Gamma\at{t}J_+f\at{t}.
\end{align}
Doing so we benefit from the following two observations. The loss
operator $J_-$ is diagonal and can hence easily be incorporated into
the homogeneous problem. The resulting solution operators take care
of the transport as before, but describe additionally the
exponential relaxation of $f_n$ to $0$ along each characteristic
line with local rate proportional to $\Gamma\at{t}$. In particular,
these solution operators preserve the non--negativity of the initial
data for all times. On the other side, the gain operator $J_+$
preserves the cone of non--negative functions, so that the
non--negativity of solutions turns out to be a direct consequence of
Duhamel's principle with respect to the gain operator.
%
%
\subsubsection*{Exponential relaxation along characteristics}
%
%
The mild solutions to the homogeneous problem corresponding to
\eqref{AuxProb.Form.2} are described by a two-parameter family of
linear operators defined by
\begin{align}
\label{AuxProb.Form.5}
\calT_{\Gamma}^-\pair{s}{t}:\sspAux\to\sspAux,\quad\quad\quad
\at{\calT_{\Gamma}^-\pair{s}{t}f}_n\deq
\exp\Bat{-u_n\int\limits_{s}^t\Gamma\at{\tilde{s}}\,\dint{\tilde{s}}}\,
\calT_{n-6}\at{t-s}f_n,
\end{align}
where $s$ and $t$ are two times with $0\leq{s}\leq{t}\leq{T}$, and
$u_n$ is the relaxation constant appearing in the loss operator.
More precisely, according to
\eqref{ApproxProb.CouplingOperator.Gain} and
\eqref{ApproxProb.CouplingOperator.Loss} we have
\begin{align}
\label{AuxProb.Form.4} %
u_2=2\beta,\qquad%
u_{n_0}=\at{\beta+1}n_0,\qquad%
\text{and}\qquad%
u_n=\at{2\beta+1}n\quad\text{ for }\quad 2<n<n_0.
\end{align}
With \eqref{AuxProb.Form.5}, each mild solution to
\eqref{AuxProb.Form.2} satisfies
\begin{align*}
f\at{t}=\calT_{\Gamma}^-\pair{0}{t}g+\int\limits_{0}^t\calT_{\Gamma}^-\pair{s}{t}h\at{s}\,\dint{s},
\end{align*}
with $h$ as in \eqref{AuxProb.Form.3}. As outlined above, the
introduction of $\calT^-_{\Gamma}\pair{s}{t}$ is motivated by
technical reasons but does not change the auxiliary problem. More
precisely, we find the following equivalence.
\begin{remark}
A mild solution $f\in\fspAux$ to Problem \ref{AuxProb.Prob}, that is
\begin{align*}
f\at{t}=\calT\at{t}g+\int\limits_0^t\calT\at{t-s}\Gamma\at{s}\at{Jf\at{s}}\,\dint{s},
\end{align*}
also satisfies
\begin{align*}
f\at{t}=\calT_{\Gamma}^-\pair{0}{t}g+\int\limits_0^t\calT_{\Gamma}^-
\pair{s}{t}\Gamma\at{s}\at{J_+f\at{s}}\,\dint{s},
\end{align*}
and vice versa.
\end{remark}
%
%
%
%
\subsubsection*{Super--solution to the auxiliary problem}
%
%
In order to establish an  existence result for mild solutions we
start by identifying a suitable super--solution $\phi$ to Problem
\ref{AuxProb.Prob} that does not depend on $t$ and $a$, and provides
suitable a priori bounds, see Lemma \ref{AuxProb.Lemma.AprioriEst}.
\begin{remark}
\label{Rem.Def.Supersol} %
Let $\phi\in{X}_{n_0}$ be defined by
\begin{align}
\notag
\phi_n\at{a}=\frac{1}{n\beta}\at{\frac{\beta}{1+\beta}}^n.
\end{align}
Then, $\phi$ is a solution to $J\phi=0$. Moreover, $\ol{f}$ with
$\ol{f}_n\pair{a}{t}=\phi_n$ for all $t\geq0$, $a\geq0$ and
$2\leq{n}\leq{n_0}$ provides a solution to the differential
equations \eqref{AuxProb.Form.1}.
\end{remark}
\begin{proof}
For $n\geq3$ we find
\begin{align*}
j_n:=\left(\beta + 1\right)n{\phi_{n}} - \beta\at{n-1}{\phi_{n-1}}=
\at{\frac{\beta+1}{\beta}}\at{\frac{\beta}{\beta+1}}^{n}-\at{\frac{\beta}{\beta+1}}^{n-1}=0,
\end{align*}
and this implies $\nat{J\phi}_{n}=j_{n+1}-j_n=0$ for all $2<n<n_0$.
Moreover, we have $\nat{J \phi}_{2}=j_3=0$ and
$\nat{J\phi}_{n_0}=-j_{n_0}=0$.
\end{proof}
We next show that a suitably chosen multiple of $\ol{f}$ provides an
upper bound for each solution to the approximate problem. For this
reason we introduce a special norm for tuples
$\triple{x_2}{...}{x_{n_0}}\in\xspAux$ by
\begin{align}
\label{Def.BNorm1}
\babs{x}:=\sup\limits_{2\leq{n}<{n_0+1}}\frac{\abs{x_n}}{\phi_n}.
\end{align}
This norm can easily be extended to states $f\in\sspAux$ via
$\bnorm{f}:=\sup_{a\geq0}\babs{f\at{a}}$, and writing
\begin{align}
\label{Def.BNorm2}
\bnorm{f}=\sup\limits_{2\leq{n}<{n_0+1}}\bpnorm{f}{n}, \qquad
\bpnorm{f}{n}:=\sup\limits_{a\geq0}\frac{\abs{f_n\at{a}}}{\phi_n}
\end{align}
we infer that
\begin{align*}
\norm{f_n}_\infty\leq\bnorm{f}\phi_n\leq{}C_\beta\bnorm{f}.
\end{align*}
Notice that the norms $\bnorm{\cdot}$ and $\norm{\cdot}_\infty$ are
equivalent on $\sspAux$ since $n_0<\infty$.
However, we prefer working with $\bnorm{\cdot}$ because only this
norm gives rise to estimates that are independent of $n_0$.
\begin{remark}
\label{AuxProb.Remark.Est1}%
For all $f\in\sspAux$ and $2\leq{n}\leq{n_0}$ we have
\begin{align*}
\norm{\at{J_+f}_n}_\infty\leq{u_n}\phi_n\bnorm{f}
\end{align*}
with $u_n$ as in \eqref{AuxProb.Form.4}.
\end{remark}
\begin{proof}
It is sufficient to consider $\bnorm{f}=1$. For $2<n<n_0$ we find
\begin{align*}
\norm{\at{J_+f}_n}_\infty
&\leq%
\at{\beta+1}\at{n+1}\norm{f_{n+1}}_\infty+
\beta\at{n-1}\norm{f_{n-1}}_\infty
\\&\leq%
\frac{1+\beta}{\beta}\at{\frac{\beta}{1+\beta}}^{n+1}+
\at{\frac{\beta}{\beta+1}}^{n-1}
=%
\at{\frac{\beta}{\beta+1}}^{n}\at{\frac{2\beta+1}{\beta}}=u_n\phi_n.
\end{align*}
Finally, the estimates
\begin{align*}
\norm{\at{J_+f}_{n_0}}_\infty
=%
\beta\at{n_0-1}\norm{f_{n_0-1}}_\infty
\leq%
\at{\frac{\beta}{\beta+1}}^{n_0-1}
=\at{\frac{\beta}{\beta+1}}^{n_0}\at{\frac{\beta+1}{\beta}}=u_{n_0}\phi_{n_0},
\end{align*}
and
\begin{align*}
\norm{\at{J_+f}_{2}}_\infty
=%
\at{1+\beta}3\norm{f_{3}}_\infty
\leq%
\frac{1+\beta}{\beta}\at{\frac{\beta}{\beta+1}}^{3}
=\at{\frac{\beta}{\beta+1}}^{2}=u_2\phi_2
\end{align*}
complete the proof.
\end{proof}
%
%
\subsubsection*{Existence of solutions to the auxiliary problem}
%
\newcommand{\itopAux}{{\mathcal{I}_{\aux}}}
Our existence and uniqueness results for the auxiliary problem are
based on Banach's contraction principle applied to the iteration
operator
\begin{align*}
\itopAux:\fspAux\to\fspAux,\quad\quad\quad
\at{\itopAux{f}}\at{t}=\calT_{\Gamma}^-\at{t}g+
\int\limits_0^t\calT_{\Gamma}^-\pair{s}{t}\Gamma\at{s}
\at{J_+f\at{s}}\,\dint{s}.
\end{align*}
We start with deriving some a priori estimates for $\itopAux$.
\begin{lemma}
\label{AuxProb.Lemma.AprioriEst}%
For each $f\in\fspAux$ we have
\begin{align}
\notag
\sup_{0\leq{t}\leq{T}}\bnorm{\itopAux{f}\at{t}}
\leq\max\{\bnorm{g},\,\sup_{0\leq{t}\leq{T}}\bnorm{f\at{t}}\}.
\end{align}
\end{lemma}
\begin{proof}
Let $u_n$ be as in \eqref{AuxProb.Form.4} and
$G\at{t}=\int\limits_0^t\Gamma\at{s}\,\dint{s}$. Definition
\eqref{AuxProb.Form.5} and Remark \ref{AuxProb.Remark.Est1} imply
\begin{align*}
\norm{\at{\itopAux{f}}_n\at{t}}_\infty
&\leq%
\phi_n\at{\exp\at{-u_nG\at{t}}\bnorm{g}+
\int\limits_0^t\exp\at{-u_nG\at{t}+u_nG\at{s}}u_n
\Gamma\at{s}\bnorm{f\at{s}}\,\dint{s}},
\end{align*}
and we conclude
\begin{align}
\notag%
\bpnorm{\at{\itopAux{f}}\at{t}}{n}
&\leq%
\exp\at{-u_nG\at{t}}\at{\bnorm{g}+
\mu\at{t}\int\limits_0^t\exp\at{u_nG\at{s}}u_n\Gamma\at{s}\,\dint{s}}
\\&=%
\label{AuxProb.Lemma.AprioriEst.Eqn2}
\exp\at{-u_nG\at{t}}\Bat{\bnorm{g}+
\mu\at{t}\at{\exp\at{u_nG\at{t}}-1}},
\end{align}
with $\mu\at{t}=\sup_{0\leq{s}\leq{t}}\bnorm{f\at{s}}$, which yields
the desired estimate.
\end{proof}
We mention that all results obtained so far in this section
immediately apply to the case $n_0=\infty$, provided that the
initial data satisfy $\bnorm{g}<\infty$.
\begin{corollary}
\label{AuxProb.Corr.Existence}%
For arbitrary $0<T<\infty$ and any initial data $g\in\sspAux{}$
there exists a unique mild solution $f\in\fspAux$ to Problem
\ref{AuxProb.Prob} which satisfies $\bnorm{f\at{t}}\leq\bnorm{g}$
for all $t\in\ccinterval{0}{T}$. Moreover, this solutions is
\begin{enumerate}
\item
non-negative provided that the initial data are non-negative, and
\item
$C^1$-regular ($L^1$-regular) if the initial data are $C^1$-regular
($L^1$-regular).
\end{enumerate}
\end{corollary}
\begin{proof}
The existence and uniqueness of a mild solution is a direct consequence of Banach's
contraction principle and Lemma \ref{AuxProb.Lemma.AprioriEst} (the
Lipschitz constant of $\itopAux$ with respect to $\bnorm{\cdot{}}$
can be read off from \eqref{AuxProb.Lemma.AprioriEst.Eqn2} with $g\equiv0$, and is
given by $1-\exp\at{-u_{n_0}G\at{t}}$). Moreover, the remaining
assertions are satisfied as $\itopAux$ respects both the
non--negativity and $L^1$-regularity.
\end{proof}
%
%
\subsection{Solutions to the approximate system}
\label{sec:Proofs.LocalSol}
%
%
In this section we are going to establish local existence and uniqueness
results for the approximate problem. To this end we consider the
solution space $\fspApp=C\at{\ccinterval{0}{T};\sspApp} $, where the
state space given by
\begin{align*}
\sspApp=\sspAux\cap{}\{f=f_n\at{a}\;:\;\norm{f_n}_{0,\,1}<\infty\text{
for all }2\leq{n}\leq{n_0}\}.
\end{align*}
Of course, the set of all \emph{admissible} states is only a proper
subset of $\sspApp$ as any reasonable solution to the approximate
problem must be non-negative and area conserving, and has to
satisfy the boundary conditions and the polyhedral formula.
\begin{definition} %
\label{appsys.Regularity}%
The state $f\in\sspApp$ is called \emph{admissible}, if it is
non-negative with $A\at{f}>0$, and satisfies the polyhedral formula
$P\at{f}=0$ as well as the boundary conditions $f_n\at{0}=0$ for
$n\geq7$.
\end{definition}
%
\subsubsection*{Iteration scheme for the approximate system}
%
We construct solutions to the approximate system as fixed points to the following iteration operator
$\itopApp:\fspApp\to\fspApp$. For given $f\in\fspApp$ let $\tilde{f}=\itopApp{f}$ be the mild solution 
to
\begin{align}
\label{AuxProb.ItProb}
\partial_t\tilde{f}_n+\at{n-6}\partial_a\tilde{f}_n&=\abs{\Gamma\nat{f}}
\at{J{f}}_n
\end{align}
with initial and boundary conditions \eqref{ApproxProb.ICandBC}.
\begin{lemma}
\label{Existence.Local.Solution1}
For all initial data $g\in\sspApp$ there exist a time $T>0$ and
a constant ${C_0}$ (both depending on $n_0$ and $g$) such
that the operator $\itopApp$ has the following properties.
\begin{enumerate}
\item
For given $f\in{\calY_\aux}$ the function $\tilde{f}=\itopApp{f}\in\calY_\aux$
is a mild solution to \eqref{AuxProb.ItProb} that satisfies the initial and 
boundary conditions \eqref{ApproxProb.ICandBC}.
\item
For $C^1$-regular data $g$ and $f$ the function $\itopApp{f}$ is
$C^1$-regular.
\item 
The set
$\fspttApp=BC\bat{\ccinterval{0}{T};\,\sspttApp}$ with
\begin{align}
\label{Existence.Local.Solution1.Def}
\sspttApp&=\left\{f\in{\sspApp}\;:\;
\Gamma_D\at{f}\geq{C_0}^{-1},\quad
\norm{f}_\infty\leq{C_0}
\right\}.
\end{align}
is invariant under the action of $\itopApp$, and the operator $f\mapsto\abs{\Gamma\at{f}}Jf$ is Lipschitz-continuous on $\sspttApp$.
\end{enumerate}
\end{lemma}
\begin{proof} At first let ${C_0}$ and $T$ be
arbitrary but fixed, and suppose that $f\at{t}\in\sspttApp$ for all
$t\in\ccinterval{0}{T}$. For the remainder of this proof $\tilde{C}$
always denotes a constant depending only on $n_0$ and $\beta$, but the value of $\tilde{C}$ may
change from line to line. From \eqref{AuxProb.ItProb} and \eqref{LinProb1.ExMildSol.Eqn1} we
conclude that
\begin{align*}
\norm{\tilde{f}\at{t}}_{\infty} \leq&
\norm{g}_{\infty}+\int\limits_0^t
\frac{{\abs{\Gamma_N\nat{f\at{s}}}}}{\Gamma_D\at{f\at{s}}}
\norm{Jf\at{s}}_{\infty}\,\dint{s} \leq
\norm{g}_{\infty}+\tilde{C}\int\limits_0^t
\frac{\norm{f\at{s}}_{\infty}^2}{\Gamma_D\at{f\at{s}}}
\,\dint{s},
\end{align*}
and due to $f\in\sspttApp$ we have
\begin{align}
\label{Existence.Local.Solution.EqnA}
\norm{\tilde{f}\at{t}}_{\infty}\leq\norm{g}_{\infty}+
\tilde{C}\,{C_0}^3\,T.
\end{align}
Moreover, since $\Gamma_D$ is linear in $f$ we find
\begin{align}
\label{Existence.Local.Solution.EqnB}
\abs{\Gamma_D\nat{\tilde{f}\at{t}}-\Gamma_D\at{\calT\at{t}g}}
&\leq%
\int\limits_0^t
\frac{{\abs{\Gamma_N\nat{f\at{s}}}}}{\Gamma_D\at{f\at{s}}}
\bigabs{\Gamma_D\at{\calT\at{t-s}Jf\at{s}}}\,\dint{s} 
\leq\tilde{C}\,{C_0}^3\,T.
\end{align}
In order to prove the invariance of $\sspttApp$ under the action of
$\itopApp_{\it}$, we choose ${C_0}$ sufficiently large such that
\begin{align*}
2\norm{g}_\infty<{C_0},\quad\Gamma_D\at{g}>3{C_0}^{-1},
\end{align*}
In addition, we choose $T$ sufficiently small with
\begin{align*}
\Gamma_D\at{\calT\at{t}g}\geq2{C_0}^{-1}\text{\;for\;all\;$0\leq{t}\leq{T}$},\qquad
\tilde{C}\,{C_0}^3\,T<\norm{g}_\infty,\qquad
\tilde{C}\,{C_0}^3\,T<C_0^{-1}.
\end{align*}
and thanks to \eqref{Existence.Local.Solution.EqnA} and \eqref{Existence.Local.Solution.EqnB} we thus find
$\tilde{f}\in\fspApp$. Finally, the claimed regularity results are provided by Lemma 
\ref{LinProb1.PropMildSol}, and the proof of the Lipschitz-continuity is straightforward.
\end{proof}
%
%
\subsubsection*{Local existence and uniqueness }
%
%
We prove the existence of \emph{admissible} solutions to the approximate system by combining the following arguments: $\at{1}$ Banach's contraction principle implies the existence of a unique fixed point $f$ for
$\itopApp$. $\at{2}$  Our results concerning the auxiliary problem provide the desired non-negativity.  $\at{3}$  The polyhedral formula is a consequence of the equation itself, and implies the conservation of area.

\begin{lemma}
\label{Existence.Local.Solution2}
Let $T$ and $C_0$ be as in Lemma \ref{Existence.Local.Solution1}, and suppose 
that the initial data $g$ are admissible. 
Then there exists a unique mild solution 
$f\in\fspApp$ to the approximate system on the time interval $\ccinterval{0}{T}$. Moreover, 
\begin{enumerate}
\item
all states $f\at{t}$ are admissible (in the sense of Definition
\ref{appsys.Regularity}) and fulfil $\bnorm{f\at{t}}\leq\bnorm{g}$,
\item
$f$ conserves the total area and the number of grains is non--increasing.
\end{enumerate}
\end{lemma}
\begin{proof} Banach's contraction principle provides the existence of 
a fixed point $f\in\fspApp$ with $\itopApp{f}=f$. 
Then we consider the auxiliary problem with prescribed
coupling weight $\abs{\Gamma\at{f\at{t}}}$, and both the positivity and
boundedness of $f$ are consequences of Corollary
\ref{AuxProb.Corr.Existence}. In particular, we find $\Gamma\at{f\at{t}}\geq0$ and this shows that $f$
is indeed a solution to the approximate system. 
In order to prove the polyhedral formula and the conservation of area we use basically the same computations as in \S\ref{sec:model}. More precisely, Lemma \ref{LinProb1.Moments} implies
\begin{align*}
\frac{\dint}{\,\dint{t}}{P\nat{f\at{t}}}-
\sum\limits_{n=2}^{n_0}\at{n-6}^2{f}\pair{0}{t}=
\Gamma\nat{f\at{t}}
{P\at{Jf\at{t}}},
\end{align*}
and with \eqref{ApproxProb.CouplingWeights} and \eqref{ApproxProb.BasicProps.Eqn2} we infer that
\begin{align*}
\frac{\dint}{\,\dint{t}}{P\nat{f\at{t}}}=
\Gamma_N\nat{f\at{t}}-
\frac{\Gamma_N\nat{f\at{t}}}{\Gamma_D\nat{f\at{t}}}
\Gamma_D\nat{f\at{t}}=0.
\end{align*}
This implies $P\nat{f\at{t}}=0$ for all $t$ due to
$P\at{g}=0$. Similarly, we find
\begin{align*}
\frac{\dint}{\,\dint{t}}A\nat{f\at{t}}=P\nat{f\at{t}}=0.
\end{align*}
Finally, exploiting
Lemma \ref{LinProb1.Moments} once again, gives 
\begin{align*}
\frac{\dint}{\,\dint{t}}
N\nat{f\at{t}}=\sum\limits_{n=2}^{5}\at{n-6}
f_n\pair{0}{t}+
\Gamma\nat{f\at{t}}
N\at{Jf\at{t}}\leq0
\end{align*}
where we used $N\at{Jf}=0$ and $f\geq0$.
\end{proof}
\begin{remark}
\label{Existence.Local.Solution.Rem} The proofs of Lemma
\ref{Existence.Local.Solution1} and Lemma 
\ref{Existence.Local.Solution2}  imply that the local solution
does exist as long as $\Gamma_D\at{f\at{t}}$ can be bounded from
below by a positive constant $d_0$. In fact, for non-negative $f$ both $\norm{f\at{t}}_\infty$ and $\Gamma_N\at{f\at{t}}$ are bounded by the
super--solution $\overline{f}_n\pair{a}{t}=\bnorm{g}\phi_n$, and thus the constant $C_0$ from \eqref{Existence.Local.Solution1.Def} can be chosen to depend on $d_0$ and $g$ only.
Consequently, each lower estimate for
$\inf_{0\leq{t}\leq{T}}{\Gamma}_D\at{f\at{t}}$ will imply the
existence and uniqueness of solutions on the interval
$\ccinterval{0}{T}$. In what follows we derive several variants of
such estimates by bounding $N\at{f\at{t}}$ from below, respectively,
which is sufficient in view of \eqref{ApproxProb.BasicProps.Eqn3}.
\end{remark}
The first lower bound for the number of grains is an elementary consequence of the von
Neumann--Mullins law. The main observation is that grains
can grow only with speed $n_0-6$, so that for compactly supported
initial data the number of grains can be bounded from below by the total area. More precisely, 
if the initial data $g$ are compactly supported in
$\{a:0\leq{a}\leq{a_0}\}$ then we have
\begin{align}
\label{FirstEstimateForN}
N\at{f\at{t}}\geq\frac{A\at{g}}{a_0+t\at{n_0-6}}.
\end{align}
This bound is, however, not optimal because it holds only for compactly supported initial data and
depends strongly on $n_0$. For this reason we finally rely on refined
estimates to be derived in Lemma \ref{Positivity.Estimates}.
%

\subsection{Tightness estimates and global existence}
\label{sec:Proofs.Tightness}

We introduce the notion of `quasi--complement', that is the number
of grains outside a bounding frame. By enlarging this frame in time,
we show that no mass runs off at infinity in finite time, and as a
consequence we establish long-time existence for the approximate
system. Moreover, these tightness estimates also apply to the
original problem with $n_0=\infty$ provided that the initial data
decay sufficiently fast with respect to both $a$ and $n$.
\par%
Within this subsection we use the following  notations. Let $f$ be a
fixed non-negative solution to the approximate system
\eqref{ApproxProb}, defined for $0\leq{t}\leq{T}<\infty$, and set
\begin{align*}
\gamma=\log(1+{1}/{\beta}),\quad{}N\at{t}=N\at{f\at{t}},\quad\Gamma\at{t}=
\Gamma\at{f\at{t}},\quad
\overline{\Gamma}\at{t}=\sup_{0\leq{s}\leq{t}}\Gamma\at{s}.
\end{align*}
Notice that the definition of $\gamma$ implies
$\beta{n}\phi_n=\exp\at{-n\gamma}$ with
$\phi=\at{\phi_n}_{2\leq{n}\leq{n_0}}$ being the super--solution
from Remark \ref{Rem.Def.Supersol}.
\par
Given $\nu\in\Nset$ and $\alpha\in\cointerval{0}{\infty}$ we define
the \emph{first quasi--complement} by
\begin{align}
\notag
{N^{\bot}}\left(t,\alpha,\nu\right) =
\sum_{n=\nu+1}^{n_0}\int\limits_{0}^{\alpha}{f_{n}\left(a,t\right)}\,\dint{a}
+
\sum_{n=2}^{n_0}\int\limits_{\alpha}^{\infty}{f_{n}\left(a,t\right)}\,\dint{a}
\end{align}
as the non--essential part of $N\left( t \right)$. In order to
simplify the notation we allow for $\nu=0$ by setting $f_1\equiv0$
and $\at{Jf}_1=0$, so that $N\at{f\at{t}}=N^\bot\triple{t}{0}{0}$.
\par
In what follows we write
$ {N^{\bot}}\left(t,\alpha\right) :=
{N^{\bot}}\left(t,\alpha,\lfloor\alpha\rfloor\right) $
if the third argument of ${N^{\bot}}$ is given by rounding down the
second one, and refer to
\begin{align}
\label{Def.SecondQuasiCompl}
M^\bot{}\pair{t}{\alpha}=
\sum_{n=\lfloor{\alpha}\rfloor+1}^{n_0}n\int\limits_0^\infty\,f_n\pair{a}{t}\,\dint{a}+
\sum_{n=2}^{n_0}\int\limits_\alpha^\infty\,af_n\pair{a}{t}\,\dint{a}
\end{align}
as the \emph{second quasi--complement}. 
\begin{remark}
\label{Rem.Def.SecondQuasiCompl}
For all $t\geq{0}$ and $\alpha\geq0$ we have
$N^\bot{}\pair{t}{\alpha}\leq{}M^\bot{}\pair{t}{\alpha}$.
\end{remark}
\begin{proof} 
The desired estimate follows immediately from the definitions provided that $\alpha\geq1$. For
$0\leq\alpha<1$ we find
\begin{align*}
N^\bot{}\pair{t}{\alpha}=N\at{t}\leq\sum_{n=2}^{n_0}n\int\limits_0^\infty\,f_n\pair{a}{t}\,\dint{a}\leq
M^\bot{}\pair{t}{\alpha}
\end{align*}
due to $\lfloor{\alpha}\rfloor=0$.
\end{proof}
%
\subsubsection*{Estimates for $N^\bot$}
%
%
In order to derive suitable estimates for the first quasi--complement
${N^{\bot}}$ we choose $\nu=\lfloor\mu\rfloor$ and $\alpha$ as
functions of time. More precisely, $\alpha$ should grow at least
linearly in $\nu$ to compensate for the transport in $a$ along
characteristic lines, and $\nu$ should grow exponentially to control
the diffusion in $n$.
\begin{lemma}
\label{ddt Nc}%
Let $\mu,\alpha:\ccinterval{0}{T}\to\cinterval{0}{\infty}$ be two smooth
functions with 
\begin{align*}
\dot{\mu}\at{t}\geq0\quad\text{and}\quad
\dot{\alpha}\at{t}\ge\max\{\lfloor\mu\at{t}\rfloor-6,0\}. 
\end{align*}
Then, we have
\begin{align}
\label{ddt Nc eq}
{N^{\bot}}\left(t,\alpha\left(t\right),\lfloor\mu\left(t\right)\rfloor\right)
\le
{N^{\bot}}\left(0,\alpha\left(0\right),\lfloor\mu\left(0\right)\rfloor\right)
+ \bnorm{g}\,\overline{\Gamma}\at{t}
\int\limits_{0}^{t}\exp\left(-\gamma\lfloor\mu\left(s\right)\rfloor\right)\alpha\left(s\right)\,\dint{s}
\end{align}
for all $t\in\ccinterval{0}{T}$.
\end{lemma}
%
%
\begin{proof}
We define
$ \nu\left(t\right)=\lfloor\mu\left(t\right)\rfloor $
as the integer part of $\mu\left(t\right)$ and denote the jump of
$\nu\left(t\right)$ by
${[\![}\nu{]\!]}\at{t}=\nu\at{t+}-\nu\at{t-}$. Obviously, we can
restrict ourselves to $\nu\at{t}\leq{n_0}$, and for simplicity we
consider classical solutions. Our results then can be generalised to
mild solutions by approximation arguments. At first we study the
case ${[\![}\nu{]\!]}\at{t}=0$. Differentiating
${N^{\bot}}\left(t,\alpha\left(t\right),\nu\left(t\right)\right)$
and using the evolution equation \eqref{ApproxProb.EvolEqn} yields
\begin{align}
\label{eq ddt Nct}
\begin{split}
\frac{\dint{}}{\dint{t}}{N^{\bot}}\left(t,\alpha,\nu\right) &=
-\sum_{n=\nu+1}^{n_0}\left(n-6\right) f_{n}\left(\alpha,t\right) +
\Gamma\at{t}
\sum_{n=\nu+1}^{n_0}\int\limits_{0}^{\alpha}{\left(J f\right)_{n}\left(a,t\right)}\,\dint{a}\\
&\,\quad +
\sum_{n=\nu+1}^{n_0}\dot{\alpha}f_{n}\left(\alpha,t\right) -
\sum_{n=2}^{n_0}\int\limits_{\alpha}^{\infty}\left(n-6\right)\partial_{a}{f_{n}\left(a,t\right)}\,\dint{a}\\
&\,\quad + \Gamma\at{t}
\int\limits_{\alpha}^{\infty}\sum_{n=2}^{n_0}{\left(J
f\right)_{n}\left(a,t\right)}\,\dint{a} -
\sum_{n=2}^{n_0}\dot{\alpha}f_{n}\left(\alpha,t\right),
\end{split}
\end{align}
where $\alpha$ and $\nu$ are shorthand for $\alpha\at{t}$ and
$\nu\at{t}$, and due to \eqref{ApproxProb.BasicProps.Eqn1} this
implies
\begin{align*}
\frac{\dint{}}{\dint{t}}{N^{\bot}}\left(t,\alpha,\nu\right) &=
\sum_{n=2}^{\nu}\left(n-6-\dot{\alpha}\right)
f_{n}\left(\alpha,t\right) + \Gamma\at{t}
\sum_{n=\nu+1}^{n_0}\int\limits_{0}^{\alpha}{\left(J
f\right)_{n}\left(a,t\right)}\,\dint{a}
\\&\leq%
\sum_{n=\nu+1}^{n_0}\Gamma\at{t}\int\limits_{0}^{\alpha}{\left(J
f\right)_{n}\left(a,t\right)}\,\dint{a}
\end{align*}
For $\nu=0$ or $\nu=1$ we find
$\frac{\dint{}}{\dint{t}}{N^{\bot}}\pair{t}{\alpha}\leq{0}$ due to
$\sum_{n=2}^{n_0}\at{Jf}_n=0$, while for $\nu>2$ the
identities \eqref{ApproxProb.CouplingOperator.Gain} and
\eqref{ApproxProb.CouplingOperator.Loss} give
\begin{align*}
\sum_{n=\nu+1}^{n_0}{\left(J f\right)_{n}\left(a,t\right)} &=
\beta\nu{f_{\nu}\left(a,t\right)} -
\left(\beta+1\right)\left(\nu+1\right){f_{\nu+1}\left(a,t\right)}
\\&\leq%
\beta\nu{f_{\nu}\left(a,t\right)}\leq\bnorm{g}\beta\nu\phi_n=
\bnorm{g}\exp\at{-\gamma\nu},
\end{align*}
where we used that $\bnorm{f\at{t}}\leq\bnorm{g}$, see Corollary
\ref{AuxProb.Corr.Existence}. Therefore,
\begin{align}
\label{ddt Nct estimated}
\frac{\dint{}}{\dint{t}}{N^{\bot}}\left(t,\alpha\at{t},\nu\at{t}\right)
&\le \overline{\Gamma}\at{t}
\,\bnorm{g}\,\exp\left(-\gamma\nu\left(t\right)\right)\alpha\left(t\right)
\end{align}
holds for all $\nu$. In the case ${[\![}\nu{]\!]}=+1$ we find
\begin{align*}
{[\![}{N^{\bot}}\left(t,\alpha,\lfloor\mu\rfloor\right){]\!]} &=
{N^{\bot}}\left(t_{+},\alpha\at{t_{+}},\lfloor\mu\rfloor\at{t_{-}}+1\right)
-
{N^{\bot}}\left(t_{-},\alpha\at{t_{-}},\lfloor\mu\rfloor\at{t_{-}}\right)
\\&= -
\int\limits_{0}^{\alpha\at{t}}f_{\lfloor\mu\left(t_{-}\right)\rfloor}\left(a,t\right)\,\dint{a}<0
\end{align*}
as an additional part in the r.h.s. of \eqref{eq ddt Nct}, so
\eqref{ddt Nct estimated} still holds in a distributional sense.
Finally, \eqref{ddt Nc eq} follows by integrating \eqref{ddt Nct
estimated}.
\end{proof}
Next we derive a decay result for $N^{\bot}\left(t,\alpha\right)$
with respect to the variable $\alpha$ which does not depend on $n_0$ but
only on the initial data $g$. These decay estimates turn out to be
the crucial ingredient  for both, the global existence proof for
solutions and the passage to the limit $n_0\to\infty$.
\begin{corollary}
\label{eq ddt Nct.Cor}%
The estimate

\begin{align*}
{N^{\bot}}\left(t,\alpha\right) &\leq
{N^{\bot}}\left(0,\alpha\exp\left(-t\right)\right) +
C\bnorm{g}\,\overline{\Gamma}\at{t}\exp\left(-\gamma\,\alpha\exp\left(-t\right)\right)
\end{align*}
holds for all $t\geq0$ and $\alpha\geq0$, where $C$ depends only on
$\beta$.
\end{corollary}
\begin{proof}
Let $\alpha_0\geq0$ be arbitrary and consider
$\alpha\left(t\right)=\mu\left(t\right)=\alpha_0\exp\left(t\right)$,
and let $\nu\left(t\right)=\lfloor\mu\left(t\right)\rfloor$. Then,
$\dot{\alpha}\at{t}=\mu\at{t}\geq{\nu\at{t}}-6$, and Lemma \ref{ddt
Nc} implies
\begin{equation*}
{N^{\bot}}\left(t,\alpha_0\exp\left(t\right)\right) \leq
{N^{\bot}}\left(0,\alpha_0\right) +
\bnorm{g}\,\overline{\Gamma}\at{t}\,\exp\at{\gamma}\,
\int\limits_{0}^{t}\exp\left(-\gamma\mu\left(s\right)\right)\alpha\left(s\right)\,\dint{s},
\end{equation*}
where we used that
$\exp\left(-\gamma\lfloor\mu\left(s\right)\rfloor\right)
\leq%
\exp\at{-\gamma\mu\at{s}+\gamma}$.
Replacing $s$ by $a=\alpha\at{s}=\alpha_0\exp\at{s}$ we find
\begin{align*}
{N^{\bot}}\left(t,\alpha_0\exp\left(t\right)\right) &\leq
{N^{\bot}}\left(0,\alpha_0\right) +
\bnorm{g}\,\overline{\Gamma}\at{t}\,\exp\at{\gamma}\,
\int\limits_{\alpha_0}^{\alpha_0\exp\at{t}}\exp\left(-\gamma\,a\right)\,\dint{a}
\\&\leq%
{N^{\bot}}\left(0,\alpha_0\right) +
\bnorm{g}\,\overline{\Gamma}\at{t}\,\frac{\exp\at{\gamma}}{\gamma}
\exp\left(-\gamma\,\alpha_0\right).
\end{align*}
Since this identity holds for arbitrary $t$ and $\alpha_0$, we can
choose $\alpha_0=\alpha\exp\at{-t}$, which gives the desired result.
\end{proof}
\begin{remark}
The proofs of Lemma \ref{ddt Nc} and Corollary \ref{eq ddt
Nct.Cor}, and hence all estimates derived  below, can be easily
generalised to the original problem with $n_0=\infty$, provided that
$\bnorm{g}<\infty$, i.e., if for $n\to\infty$ the term
$\norm{g_n}_\infty$ decays at least as fast as
$n^{-1}\exp\at{-\gamma{n}}$.
\end{remark}
%
\subsubsection*{Estimates for $M^\bot$}
%
%
Here we exploit the properties of the $N^\bot$ and
that the decay behaviour of $M^\bot$ with
respect to $\alpha$ is controlled by $t$ and the first
quasi--complement of the initial data.
\begin{lemma}
\label{Lemma.NoRunOff} %
We have
\begin{align*}
M^\bot{}\pair{t}{\alpha}
&\le{C\exp\at{t}}\at{1+\bnorm{g}+\bnorm{g}\,\overline{\Gamma}\at{t}}{\calN^\bot_0}\at{\alpha\exp\at{-t}}
\end{align*}
for all $t\geq{0}$, $\alpha\geq0$, where ${\calN^\bot_0}$ depends
only on the initial quasi--complement via
\begin{align}
\label{Lemma.NoRunOff.EqnA} %
{\calN^\bot_0}\at{\alpha} &=
\at{\alpha+1}\,\exp\bat{-\gamma{\alpha}}\;+\;\alpha\,{N^{\bot}}\left(0,\alpha\right)
\;+\;
\int\limits_{\alpha}^{\infty}{N^{\bot}}\left(0,a\right)\,\dint{a},
\end{align}
and $C$ depends only on $\beta$.
\end{lemma}
\begin{proof}
Notice that
\begin{equation}
\label{Lemma.NoRunOff.Eqn1} %
\sum_{n=2}^{n_0}\int\limits_{\alpha}^{\infty}a{f_{n}\left(a,t\right)}\,\dint{a}
=
\sum_{n=2}^{n_0}\int\limits_{\alpha}^{\infty}\int\limits_{a}^{\infty}f_{n}\left(s,t\right)\,\dint{s}\,\dint{a}
+\alpha\sum_{n=2}^{n_0}\int\limits_{\alpha}^{\infty}{f_{n}\left(a,t\right)}\,\dint{a}
\end{equation}
via an integration by parts. Corollary \ref{eq ddt Nct.Cor} yields
\begin{equation}
\label{Lemma.NoRunOff.Eqn2} %
\sum_{n=2}^{n_0}\int\limits_{\alpha}^{\infty}{f_{n}\left(a,t\right)}\,\dint{a}
\leq {N^{\bot}}\left(0,\alpha\exp\left(-t\right)\right)
+\,C\bnorm{g}\,\overline{\Gamma}\at{t}\,\exp\left(-\gamma\alpha\exp\left(-t\right)\right),
\end{equation}
and this bounds the second term on the r.h.s of \eqref{Lemma.NoRunOff.Eqn1}. Moreover, integrating \eqref{Lemma.NoRunOff.Eqn2} we find
\begin{equation}
\label{Lemma.NoRunOff.Eqn3} %
\begin{split}
\sum_{n=2}^{n_0}\int\limits_{\alpha}^{\infty}\int\limits_{a}^{\infty}f_{n}\left(s,t\right)\,\dint{s}\,\dint{a}
\le&+%
\exp\at{t}\int\limits_{\alpha\exp\at{-t}}^{\infty}{N^{\bot}}\left(0,a\right)\,\dint{a}
\\&+%
\frac{C\bnorm{g}\,\overline{\Gamma}\at{t}\exp\at{t}}{\gamma}%
\exp\at{-\gamma{\alpha}\exp\left(-t\right)},
\end{split}
\end{equation}
and combining \eqref{Lemma.NoRunOff.Eqn2} with \eqref{Lemma.NoRunOff.Eqn3} we conclude that
\begin{align*}
\sum_{n=2}^{n_0}\int\limits_{\alpha}^{\infty}a{f_{n}\left(a,t\right)}\,\dint{a}
&\le{C\exp\at{t}}\at{1+\bnorm{g}+\bnorm{g}\,\overline{\Gamma}\at{t}}{\calN^\bot_0}\at{\alpha\exp\at{-t}}.
\end{align*}
To derive the remaining estimate we split the integration with
respect to $a$ at $a=n$, and obtain
\begin{equation*}
\begin{split}
\sum_{n=\lfloor{\alpha}\rfloor+1}^{n_0}n\int\limits_{0}^{\infty}{f_{n}\left(a,t\right)}\,\dint{a}
&\leq
\sum_{n=\lfloor{\alpha}\rfloor+1}^{n_0}n\int\limits_{0}^{n}{f_{n}\left(a,t\right)}\,\dint{a}
+
\sum_{n=\lfloor{\alpha}\rfloor+1}^{n_0}\int\limits_{n}^{\infty}a{f_{n}\left(a,t\right)}\,\dint{a}
\\&\leq
\sum_{n=\lfloor{\alpha}\rfloor+1}^{n_0}n\int\limits_{0}^{n}{f_{n}\left(a,t\right)}\,\dint{a}
+
\sum_{n=2}^{n_0}\int\limits_{\alpha}^{\infty}a{f_{n}\left(a,t\right)}\,\dint{a}.
\end{split}
\end{equation*}
We have already bounded the second term on the r.h.s., while the
first one can be estimated by using
$f_n\pair{a}{t}\leq\bnorm{g}\phi_n=\bnorm{g}\exp\at{-\gamma{n}}/\at{n\beta}$.
This gives
\begin{align*}
\sum_{n=\lfloor{\alpha}\rfloor+1}^{n_0}n\int\limits_{0}^{n}{f_{n}\left(a,t\right)}\,\dint{a}
&\le%
\frac{\bnorm{g}}{\beta}\sum_{n=\lfloor{\alpha}\rfloor+1}^{n_0}n\exp\left(-\gamma n\right)
\leq%
{}C\,\bnorm{g}\int\limits_{\alpha}^\infty\,a\exp\at{-\gamma{a}}\,\dint{a}
\\&\leq%
{}C\,\bnorm{g}\int\limits_{\alpha\exp\at{-t}}^\infty\,a\exp\at{-\gamma{a}}\,\dint{a}
\leq%
C\,\bnorm{g}{\calN^\bot_0}\at{\alpha\exp\at{-t}}
\end{align*}
and the proof is complete.
\end{proof}
%
%
\subsubsection*{Positivity of numbers of grains and global existence}
%
%
The initial data $g$ are called \emph{rapidly decreasing} if there
exists two constants $d_0$ and $D_0$ such that
\begin{align}
\label{Thightness.More.Eqn1}%
M^\bot{}\pair{0}{\alpha}\leq\,D_0\exp\at{-d_0\,\alpha}<\infty
\end{align}
for all $\alpha\geq0$. Notice that \eqref{Thightness.More.Eqn1}
implies that both $N^\bot{}\pair{0}{\alpha}$ and
${\calN^\bot_0}\at{\alpha}$
decay exponentially with respect to
$\alpha$. This follows from Remark \ref{Rem.Def.SecondQuasiCompl} and Definition \eqref{Lemma.NoRunOff.EqnA}.
\begin{remark}
Suppose that the initial data decay exponentially with respect to both
$a$ and $n$, this means there constants $\tilde{d}_0$ and
$\tilde{D}_0$ such that
$g_n\at{a}\leq\tilde{D}_0\exp\nat{-\tilde{d}_0\at{a+n}}$. Then there
exists a suitable choice of $d_0$ and $D_0$ such that
\eqref{Thightness.More.Eqn1} is satisfied.
\end{remark}
For rapidly decreasing initial data we can estimate
$M^\bot\pair{t}{\alpha}$ for arbitrary $t$ and $\alpha$ by means of
Lemma \ref{Lemma.NoRunOff}.
\begin{corollary}
\label{Thightness.More}%
Suppose that the initial data $g$ are rapidly decreasing. Then, for
each $t\geq0$ there exists constants $d_t$ and $D_t$ such that
\begin{align*}
M^\bot\pair{t}{\alpha}\leq
D_t\at{1+\overline{\Gamma}\at{t}}\exp\at{-d_t\,\alpha}
\end{align*}
holds for all $\alpha\geq0$. These constants depend on $\beta$ and
$g$, but not on $\overline{\Gamma}\at{t}$ or $n_0$.
\end{corollary}
Recall that the lower bound for $N\at{t}$ from \eqref{FirstEstimateForN}
is not optimal, as it is restricted to compactly supported initial data and depends on $n_0$.
Here we derive a better result by
exploiting the tightness estimates.
\begin{lemma}
\label{Positivity.Estimates}
For rapidly decreasing initial  data $g$ and each $t\geq0$ there
exists a constant $C_t>0$ such that
\begin{align*}
N\at{t}\geq{C_t},\quad\overline{\Gamma}\at{t}\leq1/C_t.
\end{align*}
In particular, this constant is
independent of $n_0$.
\end{lemma}
\begin{proof}
Within this proof let $C_t$ denote an arbitrary constant that
depends only on $t$, $\beta$, and $g$. According to Remark
\ref{AppSys.J.Props} and the monotonicity of $N\at{t}$ we have
\begin{align}
\label{Positivity.Estimates.Eqn1}
\ol{\Gamma}\at{t}=\sup\limits_{0\leq{s}\leq{t}}\Gamma\at{s}=
\sup\limits_{0\leq{s}\leq{t}}\frac{{c}\bnorm{g}}{N\at{s}}
\leq\frac{{c}\bnorm{g}}{N\at{t}}.
\end{align}
Moreover, the conservation of area implies
\begin{align*}
A\at{g}=A\at{f\at{t}}=%
\sum_{n=2}^{n_0}\Bat{\int\limits_0^{\alpha}\,a\,f_n\pair{a}{t}\,\dint{a}+
\int\limits_{\alpha}^\infty\,a\,f_n\pair{a}{t}\,\dint{a}}
\leq{}%
\alpha{N\at{t}}+M^\bot\pair{t}{\alpha}
\end{align*}
for all $\alpha\geq0$, so we infer from Corollary \ref{Thightness.More} that
\begin{align*}
\alpha{}N\at{t}\geq
\at{A\at{g}-D_t\at{1+\overline{\Gamma}\at{t}}\exp\at{-d_t\alpha}}
\geq%
\at{A\at{g}-C_t\at{1+N\at{t}^{-1}}\exp\at{-d_t\alpha}}.
\end{align*}
Now we choose $\alpha=\alpha_t$ such that
\begin{math}
N\at{t}\alpha_t=\tfrac{1}{2}A\at{g}=c_t,
\end{math}
that means
\begin{align*}
\alpha_t={C_t}\at{1+\ln\at{1+N\at{t}^{-1}}},
\end{align*}
and hence we find
\begin{align*}
{C_t}\at{1+\ln\at{1+N\at{t}^{-1}}}N\at{t}=1.
\end{align*}
This estimate yields the existence of a constant $C_t$ with
$N\at{t}>C_t$, and thanks to \eqref{Positivity.Estimates.Eqn1} this
provides also an corresponding upper bound for $\ol{\Gamma}\at{t}$.
\end{proof}
\begin{corollary}
\label{Positivity.LongTimeExist}
Suppose that the initial data $g$ are rapidly decreasing. Then the
solution from Lemma \ref{Existence.Local.Solution2} does
exist for all times $t\geq0$.
\end{corollary}
\begin{proof}
Lemma \ref{Positivity.Estimates} provides a priori estimate for both
$N\at{t}$ and $\overline{\Gamma}\at{t}$ for arbitrary large $t$.
Therefore, the local solution to the approximate system exists for
all times, compare Remark \ref{Existence.Local.Solution.Rem}.
\end{proof}
%
%
%
%
\subsection{Passage to the limit $n_0\to\infty$}
\label{sec:Proofs.ApproxLimit}
%
In this section we consider a fixed final time $0<T<\infty$ and pass
to the limit $n_0\to\infty$. Consequently, from now on $\at{i}$ the
formulas for $Jf$ and $\Gamma\at{f}=\Gamma_N\at{f}/\Gamma_D\at{f}$
refer to \eqref{intro.DefJ} and \eqref{intro.DefGa}, the definitions
for the original problem with $n_0=\infty$, and $\at{ii}$ the number
of grains $N\at{f}$, the area $A\at{f}$, and the polyhedral defect
$P\at{f}$ are given by \eqref{intro.DefN}, \eqref{intro.DefA}, and
\eqref{intro.DefP}, respectively.
%
%
\subsubsection*{Existence of solutions}
%
The solution space $\fspInf$ for Problem
\ref{TheInfiniteProblem} is given $\fspInf=C\at{\cointerval{0}{T};\sspInf}$
for some $0<T<\infty$ with state space
\begin{align*}
\sspInf:=
\left\{f\in{}BC\at{\cointerval{0}{\infty};X}\;:\;\bnorm{f}+\lnorm{f}<\infty\right\},
\end{align*}
where $\xspInf\cong\Rset^{\{n\geq2\}}$ abbreviates the space of real-valued
series $\quadruple{x_2}{...}{x_n}{...}$. The norm $\bnorm{\cdot}$ is defined analogously to 
\eqref{Def.BNorm1} and \eqref{Def.BNorm2} (with $n_0=\infty$), and $\lnorm{\cdot}$ reads
\begin{align*}
\lnorm{f}:=\sum\limits_{n\geq2}\int\limits_{0}^\infty\at{n+a}\abs{f_n\at{a}}\,\dint{a}.
\end{align*}
In what follows we call a state $f\in{\sspInf}$
\emph{admissible}, if $f$ is non--negative with $A\at{f}>0$, and
satisfies the polyhedral formula $P\at{f}=0$ and the boundary
conditions $f_n\at{a}=0$ for $n>6$.
\bigpar%
Our strategy for constructing mild solutions to Problem
\ref{TheInfiniteProblem} is rather straightforward. For fixed
$0<T<\infty$ and given admissible initial data $g\in\sspInf$ we
consider a sequence of corresponding solutions to the approximate
problem with increasing parameter $n_0$, and aim to show the
existence of an reasonable limit in $\fspInf$. More precisely,
within this section we consider the functions
\begin{align*}
g^\upidx{n_0}\in\sspInf,\quad f^\upidx{n_0}\in\fspInf,\quad
\Gamma^\upidx{n_0}\in{}C\at{\ccinterval{0}{T}},\quad
N^\upidx{n_0}\in{}C\at{\ccinterval{0}{T}},\quad
M^{\bot,\,\upidx{n_0}}\in{}C\at{\ccinterval{0}{T}\times\ccinterval{0}{T}}.
\end{align*}
which are defined for $n_0>6$ as follows.
\begin{enumerate}
\item For all $a\geq0$ and $0\leq{t}\leq{T}$ we have
\begin{align*}
g^\upidx{n_0}_n\at{a}=\left\{%
\begin{array}{ll}%
g_n\at{a}&\text{for }2\leq{n}\leq{n_0},\\0&\text{for\;}n>n_0,
\end{array}
\right.\quad\quad
f^\upidx{n_0}_n\pair{a}{t}=\left\{%
\begin{array}{ll}%
f_n\pair{a}{t}&\text{for }2\leq{n}\leq{n_0},\\0&\text{for\;}n>n_0.
\end{array}
\right.
\end{align*}
\item
For each $n_0$ the component functions $f^\upidx{n_0}_2,...,f^\upidx{n_0}_{n_0}$
are the unique mild solution for the approximate Problem
\ref{ApproxProb} with initial data given by
$g^\upidx{n_0}_2,...,g^\upidx{n_0}_{n_0}$.
\item
$\Gamma^\upidx{n_0}$ and $N^\upidx{n_0}$ are the corresponding
coupling weight and number of grains, respectively.
\item
$M^{\bot,\,\upidx{n_0}}$ is the second quasi--complement from
\eqref{Def.SecondQuasiCompl}.
\end{enumerate}
Recall that Lemma \ref{Existence.Local.Solution1} combined
with Corollary \ref{Positivity.LongTimeExist} provide that all these
functions are well defined, and that each $\Gamma^\upidx{n_0}$ is
positive and bounded.
\bigpar
It is natural to suppose the initial data to be admissible, but we
need a bit more for our subsequent analysis. In order to establish
compactness in $\fspInf$ and to control the tail behaviour (w.r.t. to
$n$ and $a$) we must assume that the initial data are moreover
\emph{regular}.
\begin{definition}
\label{InfSys.Rem.Regularity}
A state $f\in\sspInf$ is called $\emph{regular}$ if
the following statements are satisfied.

\begin{enumerate}
\item $f$ is equi-continuous with respect to
$\bnorm{\cdot}$, i.e., for each $\eps>0$ there exists
$\ol{\delta}=\ol{\delta}\at{\eps}$ such that
\begin{align*}
\bnorm{\calS^+_{\delta}{f}-{f}}\leq\eps\quad\text{for
all}\quad0\leq\delta\leq\ol{\delta},
\end{align*}
where $\calS^+_{\delta}$ is the left-shift operator defined by
$\calS^+_{\delta}f\at{a}=f\nat{a-\delta}$ for $a\geq\delta$ and
$\calS^+_{\delta}f\at{a}=0$ else.
\item
$f$ is rapidly decreasing w.r.t $n$ and $a$, that means there exist
two constants $d$ and $D$ such that
\begin{align*}
M^\bot\pair{f}{\alpha}\leq{D}\exp\at{-d\alpha}
\end{align*}
for all $\alpha\geq0$, where
\begin{align*}
M^\bot\pair{f}{\alpha}:=\sum_{n\geq2}\int\limits_\alpha^\infty\,af_n\at{a}\,\dint{a}+
\sum_{n\geq\lfloor{\alpha}\rfloor+1}n\int\limits_0^\infty\,f_n\at{a}\,\dint{a}
\end{align*}
is defined in line with \eqref{Def.SecondQuasiCompl}.
\end{enumerate}
\end{definition}
We proceed with some remarks concerning the regularity of the
initial data $g$. $\at{i}$ The first condition is apparently
satisfied if $g$ is differentiable w.r.t $a$ with
$\bnorm{\partial_ag}<\infty$. $\at{ii}$ Concerning the second
condition recall that $\bnorm{g}<\infty$ already implies an
exponential decay w.r.t. $n$. $\at{iii}$ Compactly supported initial
data with $g_n\at{a}=0$ for large $n$ and $a$ are rapidly decreasing.
\begin{assumption}
From now on we suppose the initial data $g\in\sspInf$ to be
admissible and regular.
\end{assumption}
As a first implication we summarise some consequences of the
tightness estimates from \S\ref{sec:Proofs.Tightness}, more precisely of
Corollary \ref{Thightness.More} and Lemma
\ref{Positivity.Estimates}.
\begin{remark}
\label{InfSys.Rem.Estimates} %
There exist constants $d$ and $D$ which depend only on the initial
data $g$ and $T$ but not on $n_0$ such that for all
$t\in\ccinterval{0}{T}$ we have
\begin{align*}
N^\upidx{n_0}\at{t}\geq{d},\qquad\Gamma^\upidx{n_0}\at{t}\leq{D},\quad
M^{\bot,\,\upidx{n_0}}\pair{t}{\alpha}\leq D\exp\at{-d\,\alpha}.
\end{align*}
\end{remark}
Our second result provides
compactness in the space of bounded and continuous functions.
\begin{lemma}
\label{InfSys.Corr.Compactness} %
For each $n$ the set $\big\{f^\upidx{n_0}_n\;:\;n_0>{n}\big\}$ is
equi-continuous in
$BC\at{\ccinterval{0}{T}\times\cointerval{0}{\infty}}$.
\end{lemma}
\begin{proof}
Within this proof let $\eps>0$ be fixed and suppose that $\delta$ is
sufficiently small. We start with the modulus of continuity in $a$
-direction as it can be estimated independently of $n$. Since the
auxiliary Problem \ref{AuxProb.Prob} is linear in $f$ and
invariant under shifts w.r.t $a$, we immediately obtain
\begin{align*}
\calS^+_{\delta}f^\upidx{n_0}\at{t}-f^\upidx{n_0}\at{t}=
\calS^+_{\delta}g^\upidx{n_0}-g^\upidx{n_0}+
\int\limits_0^t\Gamma^\upidx{n_0}\at{s}
\calT\at{t-s}J\at{\calS^+_{\delta}f^\upidx{n_0}\at{s}-f^\upidx{n_0}\at{s}}\,\dint{s},
\end{align*}
and Corollary \ref{AuxProb.Corr.Existence} provides
\begin{align}
\notag
\bnorm{\calS^+_{\delta}f^\upidx{n_0}\at{t}-f^\upidx{n_0}\at{t}}
\leq\bnorm{\calS^+_{\delta}g^\upidx{n_0}-g^\upidx{n_0}}\leq{\eps}.
\end{align}
Therefore, all function $f^\upidx{n_0}_n\at{t}$ are equi-continuous
in $a$-direction (uniformly in $n$, $n_0$ and $t$). Now let $n$ be
fixed and consider $n_0>n$. The first part of this proof implies
\begin{align*}
\norm{\bat{\calT\nat{\delta}f^\upidx{n_0}\at{t}}_n-f^\upidx{n_0}_n\at{t}}_\infty\leq\eps
\end{align*}
for all $0\leq\delta\leq\ol{\delta}$. Note that here $\ol{\delta}$
is expected to depend also on $n$ as the transport velocity
increases with $n$. Moreover, by construction we have
\begin{align*}
f^\upidx{n_0}\nat{t+\delta}=\calT\nat{\delta}f^\upidx{n_0}\at{t}+
\int\limits_0^{\delta}
\Gamma^\upidx{n_0}\at{t+s}\calT\nat{\delta-s}Jf^\upidx{n_0}\at{t+s}\,\dint{s},
\end{align*}
and since $\Gamma^\upidx{n_0}\at{t}$ is uniformly bounded, see
Remark \ref{InfSys.Rem.Estimates}, we infer that
\begin{align*}
\norm{f^\upidx{n_0}_n\nat{t+\delta}-\calT\nat{\delta}f^\upidx{n_0}_n\at{t}}_\infty\leq
C\,\int\limits_0^{\delta}\norm{\bat{Jf^\upidx{n_0}\at{t+s}}_n}_\infty\,\dint{s}\leq{C_n}\delta,
\end{align*}
with constant $C_n$ depending on $n$. Finally, we have shown that
\begin{align*}
\norm{f^\upidx{n_0}_n\nat{t+\delta}-f^\upidx{n_0}_n\nat{t}}_\infty\leq\eps
\end{align*}
for all sufficiently small $\delta$, and the proof is complete.
\end{proof}
We are now able to prove our main result, which was already stated
in \S\ref{sec:model}, Theorem \ref{TheMainTheorem}.
\begin{lemma}
\label{InfSys.Lemma.Existence} There exists a mild solution
$f\in\fspInf$ to Problem \ref{TheInfiniteProblem}, which has the
following properties.
\begin{enumerate}
\item
$f$ is non--negative with $\bnorm{f\at{t}}\leq\bnorm{g}$
\item
all states $f\at{t}$ are admissible, so that $P\at{f\at{t}}=0$,
\item all states are regular in the sense of Definition
\ref{InfSys.Rem.Regularity},
\item
$f$ conserves the area with non--increasing number of grains.
\end{enumerate}
\end{lemma}
\begin{proof}
Let $\Omega=\ccinterval{0}{T}\times\cointerval{0}{\infty}$. The
Arzela--Ascoli theorem implies that any equi-continuous subset of
$BC\at{\Omega}$ has a subsequence that converges to a bounded and
continuous limit function on $\Omega$, and that this convergence is
uniform on each compact subset of $\Omega$. Thanks to Lemma
\ref{InfSys.Corr.Compactness} there exists a sequence $\at{n_k}_k$
with $n_k\to\infty$ such that $f_2^\upidx{n_k}$ converges to a limit
$f_2$ locally uniform in $BC\at{\Omega}$. Moreover, passing to a
suitably chosen subsequence, still denoted by $\at{n_k}_k$, we can
assume that $f_3^\upidx{n_k}\to{f_3}$ for some
$f_3\in{BC}\at{\Omega}$. Iterating this argument, and using the
usual diagonal trick, we finally find a subsequence along with limit
functions $f=\quadruple{f_2}{...}{f_n}{...}$ such that
\begin{align*}
f_n^\upidx{n_k}\xrightarrow{k\to\infty}f_n\quad\text{locally uniform in $BC\at{\Omega}$}.
\end{align*}
for all $n\geq2$. By construction, the limit $f$ is non--negative
and satisfies $\bnorm{f\at{t}}\leq\bnorm{g}$ for all $t$. The
uniform tightness estimates from Remark \ref{InfSys.Rem.Estimates}
imply the $L^1$-regularity of $f$, and hence $f\in\fspInf$, as well
as $M^\bot\at{f\at{t}}\leq{D}\exp\at{-dt}$. Moreover, since these
tightness estimates control the tail behaviour with respect to $n$
and $a$ we infer that
\begin{align*}
N\bat{f^\upidx{n_k}\at{t}}\xrightarrow{k\to\infty}N\at{f\at{t}},\quad
\Gamma\bat{f^\upidx{n_k}\at{t}}\xrightarrow{k\to\infty}\Gamma\at{f\at{t}},\quad
\end{align*}
and
\begin{align*}
P\bat{f^\upidx{n_k}\at{t}}\xrightarrow{k\to\infty}P\at{f\at{t}},\quad
A\bat{f^\upidx{n_k}\at{t}}\xrightarrow{k\to\infty}N\at{f\at{t}}.
\end{align*}
From this we conclude that $f$ is in fact a mild solution to Problem
\ref{TheInfiniteProblem}, and the remaining assertions concerning
$A\at{f\at{t}}$, $N\at{f\at{t}}$ and $P\at{f\at{t}}$ follow from the
corresponding properties of $f^\upidx{n_k}$.
\end{proof}
%
%
%
%
%
\subsubsection*{Remarks on well-posedness}
%
%
To establish well-posedness for our kinetic model we must prove
both, the uniqueness of solutions and the continuous dependence on
the initial data. This can done by means of energy methods, if the
`energy distance' of two states $f$ and $\tilde{f}$ is defined via
\begin{align*}
\calE\npair{f}{\tilde{f}} = \sum_{n\ge2}
n\int_{0}^{\infty}\exp\at{-a}\bat{f_n\at{a}-\tilde{f}_n\at{a}}^2\,\dint{a}
+ \bat{N\at{f}-N\nat{\tilde{f}}}^2.
\end{align*}
The main result on $\calE$, which in turn implies uniqueness and
continuous dependence for solutions to Problem
\ref{TheInfiniteProblem}, can be stated as follows.
\begin{lemma} For all $0<T<\infty$ and any
two mild solutions $f$ and $\tilde{f}$ from Lemma
\ref{InfSys.Lemma.Existence} we have
\begin{equation}
\label{Existence:energy of difference}
\calE\npair{f\at{t}}{\tilde{f}\at{t}}\leq\exp\at{Ct}E\npair{g}{\tilde{g}}
\end{equation}
for all $0\leq{t}\leq{T}$, where the constant $C$ depends only on
$T$ and the initial data $g=f\at{0}$, $\tilde{g}=\tilde{f}\at{0}$.
\end{lemma}
The proof of \eqref{Existence:energy of difference} relies on
careful estimates for
$\frac{\dint{}}{\dint{t}}\calE\npair{f\at{t}}{\tilde{f}\at{t}}$.
Since a concise presentation would involve lengthy computations, we
omit the details here and refer the reader to \cite{Hen07}.
%

\section*{Acknowledgments}
%
RH and BN acknowledge support by the Deutsche
Forschungsgemeinschaft through the Priority Program 1095 {\it
Analysis, Modeling, and Simulation of Multiscale Problems} within
the project {\it Homogenization of many--particle systems} at the
Humboldt--Universit{\"a}t zu Berlin. 
JJLV was supported through the Alexander-von-Humboldt foundation and DGES
Grant MTM2007-61755.

%
%
\providecommand{\bysame}{\leavevmode\hbox to3em{\hrulefill}\thinspace}
\providecommand{\MR}{\relax\ifhmode\unskip\space\fi MR }
\providecommand{\MRhref}[2]{%
  \href{http://www.ams.org/mathscinet-getitem?mr=#1}{#2}
}
\providecommand{\href}[2]{#2}

\end{document}